\documentclass[11pt,letterpaper]{amsart}

\usepackage[T1]{fontenc}
\usepackage{lmodern,microtype}
\usepackage{amsmath,amssymb,amsthm,mathtools}
\usepackage[shortlabels]{enumitem}
\usepackage{booktabs,array}
\usepackage{tikz}
\usepackage{bm}
\usetikzlibrary{arrows.meta,calc,decorations.pathreplacing}
\usepackage[unicode=true,colorlinks=true,linkcolor=blue,citecolor=blue,urlcolor=blue,hypertexnames=false]{hyperref}
\numberwithin{equation}{section}
\newtheorem{theorem}{Theorem}[section]
\newtheorem{proposition}[theorem]{Proposition}
\newtheorem{lemma}[theorem]{Lemma}
\newtheorem{corollary}[theorem]{Corollary}
\newtheorem{question}[theorem]{Question}
\theoremstyle{definition}
\newtheorem{definition}[theorem]{Definition}

\theoremstyle{remark}
\newtheorem{remark}[theorem]{Remark}
\DeclareMathOperator{\Fav}{Fav}
\DeclareMathOperator{\dist}{dist}
\DeclareMathOperator{\diam}{diam}

\DeclareMathOperator{\Disc}{Disc}

\newcommand{\R}{\mathbb R}
\newcommand{\C}{\mathbb C}
\newcommand{\Q}{\mathbb Q}
\newcommand{\Z}{\mathbb Z}
\newcommand{\N}{\mathbb N}
\newcommand{\T}{\mathbb T}

\newcommand{\cA}{\mathcal A}
\newcommand{\cB}{\mathcal B}
\newcommand{\cF}{\mathcal F}
\newcommand{\cH}{\mathcal H}
\newcommand{\cW}{\mathcal W}
\newcommand{\one}{\mathbf 1}
\newcommand{\ee}[1]{e^{2\pi i#1}}
\newcommand{\E}{\mathbb E}
\newcommand{\Prob}{\mathbb P}
\newcommand{\Nbd}[2]{\mathcal N_{#1}(#2)}

\allowdisplaybreaks
\definecolor{kaining}{HTML}{C2410C}
\definecolor{layna}{HTML}{1F7A5C}
\definecolor{tori}{HTML}{BF42A9}
\definecolor{caleb}{HTML}{2346D4}

\usepackage[unicode=true,colorlinks=true,linkcolor=blue,citecolor=blue,urlcolor=blue,hypertexnames=false]{hyperref}
\title[Uniform projections and Vitushkin inequalities]{Nearly Sharp Uniform Projection Bounds\\and Superlinear Vitushkin Inequalities\\for Deterministic Self-Similar Sets}
\author[K. ~Jia]{Kaining Jia}
\address{Department of Mathematics, The University of British Columbia, Vancouver, BC, Canada}
\email{kaining.jia@math.ubc.ca}
\author[L. ~Mangiapanello]{Layna Mangiapanello}
\address{Department of Mathematics, Missouri State University, 901 S. National Ave., Springfield, MO 65897, United States}
\email{lm943s@missouristate.edu}
\author[C.~Marshall]{Caleb Marshall}
\address{Department of Mathematics, The University of Toronto, 40 St. George Street, Room 6290, Toronto, ON M5S 2E4}
\email{caleb.marshall@utoronto.ca}
\author[V. ~Risner]{Victoria Risner}
\address{Department of Mathematics, Missouri State University, 901 S. National Ave., Springfield, MO 65897, United States}
\email{Vr33s@missouristate.edu}
\date{\today}
\dedicatory{For Ruthie}
\hypersetup{pdftitle={Nearly Sharp Uniform Projection Bounds and Superlinear Vitushkin Inequalities for Deterministic Self-Similar Sets},pdfauthor={Kaining Jia, Layna Mangiapanello, Caleb Marshall, Victoria Risner}}

\begin{document}
\begin{abstract}
For all integers $1\le k<d$ and every $\alpha\in(0,1)$, we construct an explicit deterministic finite homogeneous iterated function system of similarities in $\R^d$ satisfying the open set condition. Its attractor $F_{k,\alpha}$ has positive and finite $k$-dimensional Hausdorff measure and satisfies
\[
 \sup_{V\in\mathcal G(d,k)}
 \cH^k\bigl(P_V(\Nbd{\delta}{F_{k,\alpha}})\bigr)
 \lesssim_{k,\alpha}(\log\delta^{-1})^{-\alpha}
\]
for all sufficiently small $\delta>0$, where $P_V$ denotes orthogonal projection onto the $k$-plane $V$. The defining system is fixed at every scale.

The construction places the product of a planar self-similar set $E_\alpha$ and the unit $(k-1)$-cube in a coordinate $(k+1)$-plane. At every $0<\delta<1$, the supremum above is comparable, with constants depending only on $k$, to the largest projected length of the planar neighbourhood, for which we prove
\[
 \sup_{0\le\theta<\pi}
 \cH^1\bigl(\pi_\theta(\Nbd{\delta}{E_\alpha})\bigr)
 \lesssim_\alpha(\log\delta^{-1})^{-\alpha}.
\]
For the planar sets $E_\alpha$, we also prove the sharp analytic capacity law
\[
 \gamma(\Nbd{\delta}{E_\alpha})
 \asymp_\alpha(\log\delta^{-1})^{-1/2}, \qquad 0<\delta\ll 1.
\]
Consequently, each $E_\alpha$ satisfies the quantitative Vitushkin inequality
\[
 \Fav(\Nbd{\delta}{E_\alpha})
 \lesssim_\alpha\gamma(\Nbd{\delta}{E_\alpha})^{2\alpha},
 \qquad \forall\delta\in(0,1),
\]
where $\Fav(K)$ denotes the Favard length of a compact planar set $K$. These comparisons are superlinear for $\alpha>1/2$, with exponents arbitrarily close to the optimal upper limit two for this class of self-similar sets. 


\end{abstract}
\maketitle
\clearpage
\begingroup
\linespread{1.08}\selectfont
\makeatletter
\def\l@section{\@tocline{1}{0.7\baselineskip}{1pc}{}{\bfseries}}
\def\l@subsection{\@tocline{2}{0.12\baselineskip}{1pc}{5pc}{}}
\makeatother
\tableofcontents
\endgroup
\clearpage

\section{Introduction}\label{sec:introduction}
For $\theta\in[0,\pi)$, let
\[
 \pi_\theta(z):=\operatorname{Re}(e^{-i\theta}z),\qquad z\in\C,
\]
denote the scalar coordinate of orthogonal projection onto the line making angle $\theta$ with the real axis. The \emph{Favard length} of a compact set $K\subset\C$ is the quantity
\begin{equation}\label{eq:favard-definition}
 \Fav(K):=\frac1\pi\int_0^\pi\cH^1(\pi_\theta(K))\,d\theta.
\end{equation}
For compact $K\subset\R^m$, we use the closed neighbourhood in the ambient space,
\begin{equation}\label{eq:neighbourhood-definition}
 \Nbd{\delta}{K}:=\{x\in\R^m:\dist(x,K)\le\delta\},
\end{equation}
so that all neighbourhoods considered below are compact. The same projection upper bounds hold for open neighbourhoods.

The Besicovitch--Federer projection theorem \cite[Theorem~18.1]{Mattila1995} states that, for a compact set with positive and finite one-dimensional Hausdorff measure, having zero Favard length is equivalent to being purely $1$-unrectifiable. Thus, for such a purely $1$-unrectifiable set, one has
\[
 \Fav(\Nbd{\delta}{K})\longrightarrow0\qquad\text{as }\delta\rightarrow0.
\]
A quantitative form of this statement asks how quickly the average projected length tends to zero. For one-dimensional self-similar sets, a natural class of estimates consists of the \emph{power laws}
\begin{equation}\label{eq:intro-powerlaw}
 \Fav(\Nbd{\delta}{K})\lesssim (\log\delta^{-1})^{-\alpha}.
\end{equation}
The power is taken in the logarithm of the spatial scale: if the common contraction factor of the defining similarities is $L^{-1}$ and $\delta=L^{-N}$, then the right side of \eqref{eq:intro-powerlaw} is comparable to $N^{-\alpha}$.

The purpose of this article is to give explicit deterministic examples for which the exponent in \eqref{eq:intro-powerlaw} can be chosen arbitrarily close to one. Moreover, our estimates are uniform in the projection direction before angular averaging. We also calculate the analytic capacity of the corresponding neighbourhoods, which allows us to compare these two quantities at every sufficiently small scale.

For a compact set $K\subset\C$, its \emph{analytic capacity} is
\begin{equation}\label{eq:intro-capacity-definition}
 \gamma(K):=\sup\bigl\{|f'(\infty)|:
 f\text{ is analytic on }\C\setminus K,\ \|f\|_\infty\le1,
 \ f(\infty)=0\bigr\},
\end{equation}
where $f'(\infty):=\lim_{z\to\infty}zf(z)$.

\subsection{Exponents arbitrarily close to one}\label{subsec:intro-family}
Our main result supplies, for each $\alpha\in(0,1)$, an explicit finite iterated function system $\cF_\alpha$. The proof uses the number of rotating similarities to control the quantitative distribution of word orientations. Increasing this number, together with the number of collinear similarities, yields every exponent in $(0,1)$.

\begin{theorem}\label{thm:arbitrary-exponent}
For every $\alpha\in(0,1)$ there is an explicitly specified finite family $\cF_\alpha$ of planar similarities, all with the same contraction factor and satisfying the strong separation condition, whose attractor $E_\alpha \subset \mathbb{R}^2$ has the following properties:
\begin{enumerate}[(1)]
 \item $E_\alpha$ is noncollinear, $\dim_HE_\alpha=1$, and
 $0<\cH^1(E_\alpha)<\infty$.
 \medskip
 \item For every sufficiently small $\delta>0$,
 \begin{equation}\label{eq:arbitrary-projection}\sup_{\theta \in [0,\pi)}\cH^1\bigl(\pi_\theta(\Nbd{\delta}{E_\alpha})\bigr)
 \lesssim_\alpha(\log\delta^{-1})^{-\alpha}.
 \end{equation}
 \item For every sufficiently small $\delta>0$, the analytic capacity satisfies
 \begin{equation}\label{eq:arbitrary-capacity}
 \gamma(\Nbd{\delta}{E_\alpha})\asymp_\alpha
 (\log\delta^{-1})^{-1/2}.
 \end{equation}
\end{enumerate}
The family $\cF_\alpha$ is fixed independently of $\delta$; however, the constants and the defining family may depend on $\alpha$.
\end{theorem}

The exponent one in Theorem~\ref{thm:arbitrary-exponent} is a natural boundary, as all of our planar sets satisfy
\begin{equation}\label{eq:intro-universal-lower}
 \Fav(\Nbd{\delta}{E_\alpha})\gtrsim_\alpha
 (\log\delta^{-1})^{-1}.
\end{equation}
This is consistent with Mattila's general lower bound \cite{Mattila1990}; we give a direct proof for our cylinder constructions in Proposition~\ref{prop:favard-lower}. The theorem supplies a different fixed finite system for each requested exponent. It does not assert that a single finite system satisfies \eqref{eq:arbitrary-projection} for every $\alpha<1$, nor does it assert the endpoint $\alpha=1$.

\subsection{A completely specified example}\label{subsec:intro-example}
Theorem~\ref{thm:arbitrary-exponent} admits a particularly simple realization with sixteen similarities whose associated projections admit a uniform power law with exponent $53/100$. The family realizing this is defined as follows.

Let
\[
 \xi:=2^{1/4},\qquad
 (v_1,v_2,v_3):=\left(-\frac12+\frac i2,\frac i2,\frac12+\frac i2\right),
\]
and consider the sixteen similarities
\begin{equation}\label{eq:intro-sixteen-maps}
 \begin{split}
 H_j(z)&:=\frac z{16}+\frac{j-6}{7},\qquad 0\le j\le12,\\
 T_\ell(z)&:=\frac{\ee{\xi^\ell}}{16}z+v_\ell,
       \qquad 1\le\ell\le3.
 \end{split}
\end{equation}
We then let $E \subset \mathbb{C}$ denote their unique nonempty compact invariant set. Every similarity has contraction factor $1/16$, and one can check explicitly that the images of the closed unit disk in $\mathbb{C}$ are all pairwise disjoint and lie strictly inside that disk, so the system satisfies the strong separation condition. In particular, the set has Hausdorff dimension one and positive finite length. We verify each of these assertions in Section~\ref{sec:explicit}.

Our main result for this example is the following.
\begin{theorem}\label{thm:explicit-main}
Let $E \subset \mathbb{C}$ be the attractor of \eqref{eq:intro-sixteen-maps}. Then there exist constants $c,C>0$ and $\delta_0\in(0,1)$ such that, for every $\delta\in(0,\delta_0)$, one has
\begin{align}
 \sup_{0\le\theta<\pi}\cH^1\bigl(\pi_\theta(\Nbd{\delta}{E})\bigr)
 &\le C(\log\delta^{-1})^{-53/100},\label{eq:explicit-main-projection}\\
 c(\log\delta^{-1})^{-1/2}
 \le\gamma(\Nbd{\delta}{E})
 &\le C(\log\delta^{-1})^{-1/2}.\label{eq:explicit-main-capacity}
\end{align}
Consequently,
\begin{equation}\label{eq:explicit-main-comparison}
 \Fav(\Nbd{\delta}{E})
 \le C\gamma(\Nbd{\delta}{E})^{53/50}.
\end{equation}
The set $E$ is purely $1$-unrectifiable and satisfies $0<\cH^1(E)<\infty$.
\end{theorem}

The exponent $53/100$ is a convenient rational choice. The proof actually gives every exponent strictly below
\begin{equation}\label{eq:explicit-beta}
 \beta_*:=\frac{\log13}{(8/3)\log16-\log13}
 =0.5311971237\ldots,
\end{equation}
with a logarithmic factor at the endpoint in the more precise estimate; see Corollary~\ref{cor:sixteen-endpoint}. In particular, \eqref{eq:explicit-main-comparison} gives
\begin{equation}\label{eq:explicit-little-o}
 \frac{\Fav(\Nbd{\delta}{E})}{\gamma(\Nbd{\delta}{E})}
 \lesssim(\log\delta^{-1})^{-3/100}\longrightarrow0
\end{equation}
as $\delta \rightarrow 0$. Our work thus furnishes explicit deterministic self-similar sets whose $\delta$-neighbourhoods satisfy, for all sufficiently small $\delta$, a stronger comparison than that predicted by the quantitative Vitushkin conjecture \cite[Eq.~(1.7)]{CT2020}, which asserts that there is an absolute constant \(C\) such that \(\operatorname{Fav}(K)\le C\gamma(K)\) for every compact set \(K\subset\mathbb C\).

The exponent in the capacity comparison can be taken arbitrarily close to two, as the following corollary shows.

\begin{corollary}\label{cor:near-quadratic}
For every $p\in(1,2)$ there is an explicit deterministic homogeneous self-similar set $E_p\subset\C$, of Hausdorff dimension one and satisfying the strong separation condition, such that
\begin{equation}\label{eq:near-quadratic}
 \Fav(\Nbd{\delta}{E_p})
 \lesssim_p\gamma(\Nbd{\delta}{E_p})^p
 \qquad \forall \delta\in(0,1).
\end{equation}
Moreover, $\gamma(\Nbd{\delta}{E_p})\to0$ as $\delta\rightarrow0$.
\end{corollary}

\begin{proof}
Set $\alpha=p/2$ in Theorem~\ref{thm:arbitrary-exponent} and apply the lower bound in \eqref{eq:arbitrary-capacity}. This proves \eqref{eq:near-quadratic} for $\delta\in(0,\delta_p)$, with some $\delta_p\in(0,1)$. For $\delta\in[\delta_p,1)$, the set $\Nbd{\delta}{E_p}$ contains a disk of radius $\delta_p$, so monotonicity gives $\gamma(\Nbd{\delta}{E_p})\ge\delta_p$. Since the construction lies in the unit disk, $\Fav(\Nbd{\delta}{E_p})\le4$. Enlarging the comparison constant proves the estimate on the remaining scales.
\end{proof}

\begin{remark}
Both Favard length and analytic capacity are homogeneous of degree one. For $D:=\diam(E_p)$, the scale-invariant form of \eqref{eq:near-quadratic} is
\[
 \frac{\Fav(\Nbd{\delta}{E_p})}{D}
 \lesssim_p\left(\frac{\gamma(\Nbd{\delta}{E_p})}{D}\right)^p.
\]
Equivalently, the unnormalized right side contains the factor $D^{1-p}$.
\end{remark}
\subsection{Uniform projections onto higher-dimensional planes}\label{subsec:intro-lift}
Fix $1\le k<d$, and write $\mathcal G(d,k)$ for the Grassmannian of $k$-dimensional linear subspaces of $\R^d$ and $P_V$ for orthogonal projection onto $V\in\mathcal G(d,k)$. We lift the planar set from Theorem~\ref{thm:arbitrary-exponent} by taking
\[
 F_{k,\alpha}
 =E_\alpha\times[0,1]^{k-1}\times\{0\}^{d-k-1}\subset\R^d,
\]
with zero-dimensional factors omitted. The planar system has $L$ similarities of ratio $L^{-1}$. Pairing each with the $L^{k-1}$ similarities that subdivide the cube gives a homogeneous system satisfying the open set condition. Its attractor $F_{k,\alpha}$ is purely $k$-unrectifiable and has positive and finite $\cH^k$-measure.

Theorem~\ref{thm:lifted} compares the largest projected $k$-volume of the lift with the largest projected length of the planar set at the same radius. For $0<\delta<1$,
\begin{equation}\label{eq:intro-lift}
\begin{aligned}
 \sup_{V\in\mathcal G(d,k)}
 \cH^k\bigl(P_V(\Nbd{\delta}{F_{k,\alpha}})\bigr)
 &\asymp_k \sup_{0\le\theta<\pi}
 \cH^1\bigl(\pi_\theta(\Nbd{\delta}{E_\alpha})\bigr)\\
 &\lesssim_{k,\alpha}(\log\delta^{-1})^{-\alpha}.
\end{aligned}
\end{equation}
Inside any $k$-plane $V$, the projected cube spans at most $k-1$ dimensions. Along a direction in $V$ perpendicular to that span, the coordinate of a point of $P_V(F_{k,\alpha})$ depends only on the planar factor. After thickening by $\delta$, the length of this set of coordinates is bounded by the length of the corresponding projection of the planar $\delta$-neighbourhood. The perpendicular sections have uniformly bounded $(k-1)$-dimensional measure, so integration gives the upper bound.

Conversely, take $V$ to be spanned by a planar line and all $k-1$ cube directions. The projection of the lifted neighbourhood then contains the product of the projected planar neighbourhood and the full unit cube. Since the largest projected length dominates its angular average, \eqref{eq:intro-universal-lower} gives a lower bound of order $(\log\delta^{-1})^{-1}$ for the supremum on the left of~\eqref{eq:intro-lift}. Consequently, an upper bound with exponent greater than one is impossible for these lifts.

\subsection{Related work: projections and capacities of self-similar sets}
\subsubsection{Deterministic projection estimates}
Peres and Solomyak \cite{PS2002} initiated the quantitative projection problem for deterministic self-similar sets. Nazarov, Peres, and Volberg \cite{NPV2010} introduced a Fourier analytic method giving a power law for the four-corner Cantor set. \L aba and Zhai \cite{LZ2010} extended it to product Cantor sets with a positive-length projection. Bond, \L aba, and Volberg \cite{BLV2014} obtained the bound $CN^{-c/\log\log N}$ for rational product Cantor sets under arithmetic hypotheses, and power laws under additional assumptions. \L aba and Marshall \cite{LabaMarshall2022} extended these estimates to further digit sets through new bounds for vanishing sums of roots of unity. Marshall \cite{Marshall2026} subsequently established power laws for a class of purely $1$-unrectifiable rational product Cantor sets in every dimension $d\ge2$, including the four-corner bound $N^{-1/5+u}$ for every $u>0$. These constructions use homotheties without rotations; see \L aba's survey \cite{Laba2015} for the Fourier and number-theoretic background.

For separated noncollinear homogeneous disk constructions without rotations, with $L$ maps of contraction factor $L^{-1}$, Bond and Volberg \cite[Theorem~1]{BondVolberg2012} proved
\begin{equation}\label{eq:bond-volberg-bound}
 \Fav(\mathcal G_N)\le C\exp\bigl(-c\sqrt{\log N}\bigr),
 \qquad N\ge2,
\end{equation}
where $\mathcal G_N$ is the $N$th iterate. No product structure is required. The formulation in \cite[Example~4 and Theorem~1.2(iii)]{Laba2015} uses similarities $z\mapsto L^{-1}z+z_j$. At scale $\delta=L^{-N}$, this gives $C\exp(-c\sqrt{\log\log(\delta^{-1})})$.

For systems with dense rotations, Ero\u glu \cite[Theorem~3.1]{Eroglu2007} obtained a direction-uniform negative power of $\log\log(\delta^{-1})$ under a Diophantine hypothesis. His argument groups relatively close cylinders and finds prefixes that orient them toward a prescribed line. We instead use a collinear subfamily and quantitative distribution of word orientations, adapting the discrepancy method of Hensley and Su \cite{HS2004} to the one-sided walk with a holding probability arising here. First-success selection then controls the covering cost through the masses of prefix-free stopped words.

\subsubsection{Prescribed and nonlinear projections}
The geometric compression has an analogue in Falconer's Venetian blind construction \cite[Section~2, Lemma~2.1]{Falconer1986}: tilted segments preserve selected projections while reducing others. Here a collinear group compresses when its supporting line is nearly parallel to the projection kernel, within a fixed strongly separated system.

Falconer's prescribed-projection theorem, underlying the digital sundial construction, realizes measurable families of shadows up to null sets in almost every direction \cite[Theorem~5.1]{Falconer1986}. Chang, McDonald, and Taylor \cite{CMT} proved a nonlinear counterpart yielding efficient coverings by translates of a $C^1$ graph with strictly monotone derivative. Chang, Fraser, McDonald, Meyer, and Taylor \cite{CFMMT2026} extended prescribed projections and efficient coverings to broader classes of maps and curves. These almost-everywhere prescriptions differ from the present bounds for every direction at each scale.

For quantitative nonlinear projection estimates, Cladek, Davey, and Taylor \cite{CDT2022} treated the Favard curve length of four-corner approximations; Davey and Taylor \cite{DT2022} quantified a Besicovitch theorem for curve projections; and Bongers and Taylor \cite{BT2023} developed energy estimates and lower bounds for transversal families, including radial visibility and Favard surface length. \L aba, McDonald, and Taylor \cite{LMT2026} established a local comparison principle transferring classical Favard bounds to smooth transversal families; their global self-similar application assumes homotheties without rotations. Li and Taylor \cite{LT2026} obtained a quantitative two-projection theorem. Taylor's survey \cite{Taylor2024} describes the connections with efficient coverings and fractal geometry.

\subsubsection{Random constructions and quantitative rectifiability}
Sharp inverse-logarithmic Favard estimates hold for several random constructions, including those of Chang, Shmerkin, and Suomala \cite{CSS2026}. The non-homogeneous random disc constructions of Chang, D\k{a}browski, and Del Nin \cite{CDD2026} also admit capacity estimates. Here every defining map is explicit and the same finite system is used at every generation; probability serves only to count symbolic addresses.

Chang, D\k{a}browski, Orponen, and Villa \cite{CDOV2024} obtained a different quantitative link with rectifiability: Favard length close to that of a segment forces approximation by a Lipschitz graph outside a small exceptional part, with polynomial control in the deficit.

\subsubsection{Analytic capacity and Vitushkin's question}
The remaining qualitative Vitushkin implication asks whether positive Favard length forces positive analytic capacity. Its quantitative form is
\begin{equation}\label{eq:vitushkin-question}\Fav(K)\lesssim\gamma(K)\qquad\text{for every compact }K\subset\C;
\end{equation}
see Chang and Tolsa \cite[Eq.~(1.7)]{CT2020}. Corollary~\ref{cor:near-quadratic} provides a strictly stronger, superlinear bound along the neighbourhoods of our explicit examples, not a proof of the universal inequality. The latter remains open. Conversely, positive analytic capacity need not imply positive Favard length, by examples of Jones--Murai \cite{JM1988} and Joyce--M\"orters \cite{JM2000}. Our limiting attractors have both quantities zero; the comparison concerns their neighbourhoods.

The square-root capacity decay is familiar from the four-corner construction \cite{MTV2003}. We establish it for the rotating systems here using Tolsa's curvature characterization \cite{Tolsa2003,Tolsa2005}. The proof applies to noncollinear homogeneous planar systems satisfying our disk separation assumptions, with exactly $L$ maps of contraction factor $L^{-1}$.

\subsection{Notation and organization}
Throughout the article, $q$ denotes the number of rotating similarities, not an ambient dimension. The planar construction takes place in $\C\cong\R^2$, and its lift lies in $\R^d$. We write $A\lesssim B$ if $A\le CB$, with a positive constant independent of the construction depth, the spatial scale, and the projection direction. Unless explicitly stated otherwise, constants may depend on the fixed iterated function system. Subscripts indicate additional dependence. The notation $A\asymp B$ means that both $A\lesssim B$ and $B\lesssim A$ hold. All logarithms are natural.

Section~\ref{sec:geometry} sets up the symbolic construction and its uniform geometry. Section~\ref{sec:mixing} proves the orientation estimate, and Section~\ref{sec:stacks} converts it into a projection bound. Section~\ref{sec:explicit} supplies the algebraic rotations and the explicit similarities. Section~\ref{sec:capacity} proves the capacity estimate. The main theorems are then assembled in Section~\ref{sec:conclusions}, where we also record the limits of the argument. Section~\ref{sec:lift} proves the higher-dimensional comparison.

\section{Symbolic cylinders and their uniform geometry}\label{sec:geometry}
Let $L\ge2$ be an integer and let $\cA$ be an alphabet with exactly $L$ symbols. We consider similarities
\begin{equation}\label{eq:general-maps}
 F_a(z)=L^{-1}\ee{\tau_a}z+b_a,
 \qquad a\in\cA,
\end{equation}
where $\tau_a\in\T:=\R/\Z$ and $b_a\in\C$. The closed unit disk is denoted by $D:=\overline{B(0,1)}$.

\begin{definition}\label{def:disk-separation}
The family $\cF=\{F_a:a\in\cA\}$ satisfies the \emph{disk separation assumptions} if the disks $F_a(D)$ are pairwise disjoint and
\begin{equation}\label{eq:strict-interior}
 R:=\max_{a\in\cA}\bigl(|b_a|+L^{-1}\bigr)<1.
\end{equation}
We then put
\begin{equation}\label{eq:separation-constant}
 g:=\min_{a\ne a'}\dist(F_a(D),F_{a'}(D))>0.
\end{equation}
\end{definition}

The strict containment in \eqref{eq:strict-interior} is convenient when passing between construction stages and neighbourhoods. It is verified directly for all of the examples below.

For a finite word $\omega=a_1\cdots a_k\in\cA^k$, set
\[
 |\omega|:=k,\qquad F_\omega:=F_{a_1}\circ\cdots\circ F_{a_k},\qquad
 D_\omega:=F_\omega(D),\qquad \ell_\omega:=L^{-k}.
\]
The empty word $\varnothing$ has length zero and represents the identity map. We also write
\begin{equation}\label{eq:orientation-word}
 F_\omega(z)=\ell_\omega\ee{\tau_\omega}z+b_\omega,
 \qquad \tau_\omega=\tau_{a_1}+\cdots+\tau_{a_k}\pmod1.
\end{equation}
In particular, $\tau_{\omega u}=\tau_\omega+\tau_u$. This additive identity uses the fact that the similarities preserve orientation.

The \emph{Cantor iterates} and their attractor are
\begin{equation}\label{eq:iterate-definition}
 E_N:=\bigcup_{\omega\in\cA^N}D_\omega,
 \qquad E:=\bigcap_{N=0}^\infty E_N.
\end{equation}
They are nested compact sets. Standard contraction theory identifies $E$ as the unique nonempty compact set satisfying $E=\bigcup_{a\in\cA}F_a(E)$.

Let $\Omega:=\cA^{\N}$ be the space of infinite words, and give each symbol mass $1/L$. We denote the resulting product probability measure by $\Prob$. For a finite word $\omega$, its symbolic cylinder is
\[
 [\omega]:=\{\boldsymbol{\omega}\in\Omega:\boldsymbol{\omega}\text{ begins with }\omega\},
 \qquad \Prob([\omega])=L^{-|\omega|}=\ell_\omega.
\]
Let $\Pi:\Omega\to E$ be the coding map, and put $\mu:=\Pi_\sharp\Prob$. Disk separation makes the coding unique and gives
\begin{equation}\label{eq:natural-cylinder-mass}
 \mu(D_\omega)=\ell_\omega.
\end{equation}

\begin{lemma}\label{lem:prefix-free}
If $\cW$ is a finite family of words and no word in $\cW$ is a proper prefix of another, then
\begin{equation}\label{eq:kraft}
 \sum_{\omega\in\cW}L^{-|\omega|}\le1.
\end{equation}
\end{lemma}
\begin{proof}
The assertion is immediate if $\cW=\varnothing$. Otherwise, choose a common terminal depth
\[
M\ge\max_{\omega\in\cW}|\omega|.
\]
A word $\omega$ has exactly $L^{M-|\omega|}$ extensions to depth $M$. Prefix freeness makes these extension sets disjoint. Their total cardinality is at most $L^M$; dividing by $L^M$ proves \eqref{eq:kraft}.
\end{proof}

\begin{lemma}\label{lem:first-difference}
Suppose $\omega,u\in\cA^N$ are distinct and their longest common prefix has length $k$. If $x\in D_\omega$ and $y\in D_u$, then
\begin{equation}\label{eq:first-difference}
 gL^{-k}\le |x-y|\le2L^{-k}.
\end{equation}
\end{lemma}
\begin{proof}
Write $\omega=va \omega'$ and $u=va'u'$ with $|v|=k$ and $a\ne a'$. Both points lie in $D_v$, whose diameter is $2L^{-k}$. Their inverse images under $F_v$ lie in the two distinct first-level disks $F_a(D)$ and $F_{a'}(D)$, whose distance is at least $g$. Scaling by $L^{-k}$ gives the lower bound.
\end{proof}

\begin{proposition}\label{prop:regularity}
There is a constant $C\ge1$, depending only on the fixed system, such that
\begin{equation}\label{eq:ahlfors-regularity}
 C^{-1}t\le\mu(B(x,t))\le Ct
 \qquad \forall x\in E,\ \forall t\in(0,\diam E].
\end{equation}
Consequently, $\dim_HE=1$ and $0<\cH^1(E)<\infty$.
\end{proposition}
\begin{proof}
For the upper bound choose $k$ so that $L^{-k}\le t<L^{-k+1}$. Distinct depth-$k$ disks are separated by at least $gL^{-k+1}$, by Lemma~\ref{lem:first-difference}. A packing argument in the plane shows that a ball of radius $t$ meets at most $C_{g,L}$ such disks. Each carries mass $L^{-k}$, so its mass is at most $C_{g,L}L^{-k}\le C_{g,L}t$.

For the lower bound, take the first $k$ for which $2L^{-k}<t$. The depth-$k$ cylinder containing $x$ lies in $B(x,t)$ and has mass $L^{-k}\ge t/(2L)$. This proves \eqref{eq:ahlfors-regularity}. The upper mass bound gives a positive lower bound for $\cH^1(E)$ by the mass distribution principle. On the other hand, the $L^k$ depth-$k$ disks cover $E$ and have total diameter $2$. Thus $\cH^1(E)\le2$, and the dimension assertion follows as well.
\end{proof}

We will also use the upper bound $\mu(B(z,t))\lesssim t$ for all centers $z\in\C$ and radii $t>0$. If such a ball meets $E$, choose $x\in E\cap B(z,t)$ and use $B(z,t)\subset B(x,2t)$. For radii exceeding $\diam E$, the bound follows from $\mu(\C)=1$ after enlarging the constant.

\subsection{Comparison with metric neighbourhoods}
\begin{lemma}\label{lem:neighbourhood-comparison}
Under the disk separation assumptions, one has
\begin{equation}\label{eq:iterate-neighbourhood-inclusions}
 E_k\subset\Nbd{2L^{-k}}{E},\qquad
 \Nbd{\delta}{E}\subset E_k
 \quad\text{if }\delta \in (0,(1-R)L^{-k}].
\end{equation}
Consequently, for all sufficiently small $\delta>0$, there are integers $k_+(\delta)$ and $k_-(\delta)$ such that
\begin{equation}\label{eq:neighbourhood-sandwich}
 E_{k_+(\delta)}\subset\Nbd{\delta}{E}\subset E_{k_-(\delta)},
 \qquad k_\pm(\delta)=\log_L(\delta^{-1})+O_{\cF}(1).
\end{equation}
\end{lemma}
\begin{proof}
If $z\in D_\omega$ with $|\omega|=k$, choose any $x\in F_\omega(E)$. Both points lie in $D_\omega$, so $|z-x|\le2L^{-k}$. This proves the first inclusion.

The first-level containment gives $E\subset RD$. If $x\in E$ has prefix $\omega$ of length $k$, then $x\in F_\omega(E)\subset F_\omega(RD)$. A point at distance at most $(1-R)L^{-k}$ from $x$ therefore lies in $D_\omega$. This proves the second inclusion. One may take
\[
 k_+(\delta)=\left\lceil\log_L(2/\delta)\right\rceil,
 \qquad k_-(\delta)=\left\lfloor\log_L((1-R)/\delta)\right\rfloor.
\]
For sufficiently small $\delta$, both are nonnegative and satisfy \eqref{eq:neighbourhood-sandwich}.
\end{proof}

This lemma will be used for both projection length and analytic capacity. It avoids estimating the change in length of a projection by enlarging every connected component separately.

\subsection{A lower bound for the Favard length}
\begin{proposition}\label{prop:favard-lower}
For every system satisfying the disk separation assumptions,
\begin{equation}\label{eq:iterate-favard-lower}
 \Fav(E_N)\gtrsim\frac1{N+1},\qquad N\ge0.
\end{equation}
The corresponding neighbourhood lower bound is \eqref{eq:intro-universal-lower}.
\end{proposition}
\begin{proof}
Set $\varepsilon=L^{-N}$ and define the projection counting function
\[
 f_{N,\theta}(x):=\sum_{\omega\in\cA^N}\one_{\pi_\theta(D_\omega)}(x).
\]
Every projected interval has length $2\varepsilon$, so
\begin{equation}\label{eq:lower-counting-mass}
 \int_\R f_{N,\theta}(x)\,dx=2.
\end{equation}
We claim that
\begin{equation}\label{eq:lower-averaged-energy}
 \frac1\pi\int_0^\pi\int_\R f_{N,\theta}(x)^2\,dx\,d\theta
 \lesssim N+1.
\end{equation}
The diagonal pairs of disks contribute exactly $2$. For two distinct disk centers $b_\omega,b_u$, let $s=|b_\omega-b_u|$. The two projected intervals can meet only when $|\pi_\theta(b_\omega-b_u)|\le2\varepsilon$. The measure of these angles is at most $C\min\{1,\varepsilon/s\}$, and their overlap length is at most $2\varepsilon$. Thus the angularly averaged overlap is at most $C\varepsilon^2/s$.

There are $(1-L^{-1})L^{2N-k}$ ordered pairs of depth-$N$ words whose longest common prefix has length $k$. Lemma~\ref{lem:first-difference} gives $s\ge gL^{-k}$. The contribution from all such pairs is therefore at most
\[
 C\varepsilon^2 L^kL^{2N-k}=C.
\]
Summing over $0\le k<N$ proves \eqref{eq:lower-averaged-energy}.

Cauchy--Schwarz in $x$ and \eqref{eq:lower-counting-mass} give
\[
 4\le\cH^1(\pi_\theta(E_N))\int_\R f_{N,\theta}^2.
\]
A further Cauchy--Schwarz inequality in $\theta$ yields, more explicitly,
\[
 2\le\frac1\pi\int_0^\pi
 \left(\cH^1(\pi_\theta(E_N))\int_\R f_{N,\theta}^2\right)^{1/2}d\theta
 \le\left(\Fav(E_N)\frac1\pi\int_0^\pi\int_\R f_{N,\theta}^2\right)^{1/2}.
\]
Consequently,
\[
 4\le\Fav(E_N)\left(\frac1\pi\int_0^\pi\int_\R f_{N,\theta}^2\right).
\]
This proves \eqref{eq:iterate-favard-lower}. The neighbourhood version follows from the first inclusion in \eqref{eq:neighbourhood-sandwich}.
\end{proof}

\section{Quantitative distribution of orientations}\label{sec:mixing}
We now impose the arithmetic condition on the rotations. Distances to the nearest integer are denoted by $\|t\|_{\T}$. For a vector $x\in\T^q$, let
\[
 \|x\|_{\T^q,\infty}:=\max_{1\le j\le q}\|x_j\|_{\T}.
\]

\begin{definition}\label{def:badly-approximable}
A vector $\tau=(\tau_1,\ldots,\tau_q)\in\T^q$ is \emph{simultaneously badly approximable} if there is a constant $b>0$ such that
\begin{equation}\label{eq:BA}
 \|k\tau\|_{\T^q,\infty}\ge b|k|^{-1/q}
 \qquad \forall k\in\Z\setminus\{0\}.
\end{equation}
\end{definition}

Let $r\ge2$, $q\ge1$, and $L=r+q$. The orientation measure associated with $r$ nonrotating symbols and $q$ rotating symbols is
\begin{equation}\label{eq:orientation-measure}
 \sigma:=\frac rL\delta_0+\frac1L\sum_{j=1}^q\delta_{\tau_j}.
\end{equation}
Here $\delta_t$ denotes unit point mass at $t\in\T$. The term $(r/L)\delta_0$ records the probability of selecting a nonrotating symbol, which leaves the accumulated orientation unchanged.

For probability measures $\mu,\nu$ on $\T$ and every Borel set $A\subset\T$, define
\[
 (\mu*\nu)(A):=\int_\T\int_\T\one_A(x+y)\,d\mu(x)\,d\nu(y),
\]
where addition is taken modulo one. For an integer $J\ge1$, we write
\[
 \sigma^{*J}:=\underbrace{\sigma*\cdots*\sigma}_{J\text{ factors}}
\]
for the $J$-fold convolution of $\sigma$ with itself. Thus $\sigma^{*J}$ is the distribution of the sum, modulo one, of $J$ independent increments with common distribution $\sigma$. Under the uniform counting measure on words, it is precisely the distribution of the orientation of a word of length $J$.

Let $\mu_{\mathbb{T}}$ denote normalized Haar measure on $\T$, so that $\mu_{\mathbb{T}}(\T)=1$. If $\nu$ is a probability measure on $\T$, define its interval discrepancy by
\[
 \Disc(\nu):=\sup_{I\subset\T}\bigl|\nu(I)-\mu_{\mathbb{T}}(I)\bigr|,
\]
where $I$ ranges over intervals on the circle. The Fourier coefficients of $\nu$ are
\[
 \widehat\nu(k):=\int_\T e^{-2\pi ikt}\,d\nu(t), \qquad \forall k \in \mathbb{Z}.
\]
With this convention, convolution satisfies
$\widehat{\mu*\nu}(k)=\widehat\mu(k)\widehat\nu(k)$; in particular,
$\widehat{\sigma^{*J}}(k)=\widehat\sigma(k)^J$.
We use the classical Erd\H{o}s--Tur\'an inequality \cite[Part~I, Theorem~III]{ErdosTuran1948}, which states that
\begin{equation}\label{eq:ET}
 \Disc(\nu)\le C\left(\frac1{H+1}
 +\sum_{k=1}^H\frac{|\widehat\nu(k)|}{k}\right),
 \qquad \forall H\in\N_{\ge1}.
\end{equation}
The probability-measure formulation needed here appears, for example, as Eq.~(9) in \cite{HS2004}. The measure $\sigma^{*J}$ is itself the equally weighted measure on a finite list of word orientations, with repetitions retained.

\begin{proposition}\label{prop:orientation-discrepancy}
Suppose that \eqref{eq:BA} holds. There is a constant $C_{\mathrm{disc}}\ge1$, depending only on $q,r,L$ and $b$, such that
\begin{equation}\label{eq:discrepancy-rate}
 \Disc(\delta_a*\sigma^{*J})\le C_{\mathrm{disc}}J^{-q/2}
 \qquad \forall a\in\T,\ \forall J\in\N_{\ge1}.
\end{equation}
Consequently, for every $\eta\in(0,1/10)$, every interval $I\subset\T$ with $\mu_{\mathbb{T}}(I)=2\eta$ has mass at least $\eta$ under $\delta_a*\sigma^{*J}$ whenever
\begin{equation}\label{eq:mixing-length}
 J\ge C_{\mathrm{mix}}\eta^{-2/q},
 \qquad C_{\mathrm{mix}}:=C_{\mathrm{disc}}^{2/q}.
\end{equation}
Both estimates are uniform in the incoming orientation $a$.
\end{proposition}

The proof follows the discrepancy method of Hensley and Su \cite{HS2004}. We give the details because the nonrotating symbols produce a holding probability $r/L$, and the remaining rotation increments occur with only one sign.

\begin{lemma}\label{lem:gaussian-packing}
Let $x_1,\ldots,x_m\in\T^q$ be pairwise $s$-separated in the maximum metric, and suppose that $\|x_j\|_{\T^q,\infty}\ge s$ for each $j$. If $s\in(0,1/2]$ and $ts^2\ge1$, then
\begin{equation}\label{eq:gaussian-packing}
 \sum_{j=1}^m e^{-t\|x_j\|_{\T^q,\infty}^2}
 \le C_q e^{-ts^2/2}.
\end{equation}
\end{lemma}
\begin{proof}
The $s/2$-balls about the points have disjoint interiors. Volume comparison in the maximum metric gives
\[
 \#\{j:hs\le\|x_j\|_{\T^q,\infty}<(h+1)s\}
 \le C_q(h+2)^q\qquad \forall h\in\N_{\ge1}.
\]
The cumulative ball count suffices here. If $(h+1)s\ge1/2$, the same inequality follows from the total volume bound for the disjoint $s/2$-balls. Since no point has distance less than $s$ from zero, the indicated shells cover all of the points. Hence
\[
 \sum_{j=1}^m e^{-t\|x_j\|_{\T^q,\infty}^2}
 \le C_q\sum_{h=1}^\infty(h+2)^q e^{-ts^2h^2}
 \le C_qe^{-ts^2/2}.
\]
For the last inequality, factor out $e^{-ts^2/2}$ and use $ts^2\ge1$ to bound the remaining series by a convergent series depending only on $q$.
\end{proof}

The separation hypothesis also gives $m\lesssim_q s^{-q}$ by volume comparison on $\T^q$. Thus the number of points and their separation cannot be chosen independently.

\begin{proof}[Proof of Proposition~\ref{prop:orientation-discrepancy}]
Fix an integer $J\ge1$. Write $p_0=r/L$, $p_j=1/L$ for $1\le j\le q$, and $\tau_0=0$. The exact identity
\[
 1-|\widehat\sigma(k)|^2
 =\sum_{i,j=0}^q p_ip_j
       \bigl(1-\cos(2\pi k(\tau_i-\tau_j))\bigr)
\]
contains only nonnegative summands.
\begin{align*}
 1-|\widehat\sigma(k)|^2
 &\ge \frac{2r}{L^2}\sum_{j=1}^q
       \bigl(1-\cos(2\pi k\tau_j)\bigr)\\
 &\ge \frac{16r}{L^2}\sum_{j=1}^q\|k\tau_j\|_\T^2,
\end{align*}
where, in the final inequality, we use that $1-\cos(2\pi t)\ge8\|t\|_\T^2$.
Using $u^{J/2}\le\exp(-(J/2)(1-u))$ for $u\in[0,1]$, we obtain, with $a_0:=8r/L^2$,
\begin{equation}\label{eq:fourier-walk-gaussian}
 |\widehat\sigma(k)|^J
 \le\exp\left(-a_0J\sum_{j=1}^q\|k\tau_j\|_\T^2\right)
 \le e^{-a_0J\|k\tau\|_{\T^q,\infty}^2}.
\end{equation}

Decrease the approximation constant $b$ if necessary so that $b\le1/2$. For each dyadic integer $Q\ge1$, the points $k\tau$, with $Q\le k<2Q$, together with the origin, are pairwise separated by
\[
 s_Q:=b(2Q)^{-1/q}.
\]
Indeed, the separation from zero follows from \eqref{eq:BA} with $k<2Q$, and separation between two different points follows by applying it to their difference, whose absolute value is less than $Q$.

For sufficiently large $J$, choose
\begin{equation}\label{eq:discrepancy-cutoff}
 H:=\left\lfloor\frac{b^q}{2}(a_0J)^{q/2}\right\rfloor\ge1.
\end{equation}
If $Q\le H$, then $2Q\le2H\le b^q(a_0J)^{q/2}$, and hence $a_0Js_Q^2\ge1$. Combining Lemma~\ref{lem:gaussian-packing} with \eqref{eq:fourier-walk-gaussian} gives
\begin{equation}\label{eq:cohort-estimate}
 \sum_{Q\le k<2Q}|\widehat\sigma(k)|^J
 \le C_q\exp\left(-\frac{a_0Jb^2}{2(2Q)^{2/q}}\right).
\end{equation}
Let $2^S\le H<2^{S+1}$. We bound the last, possibly incomplete, dyadic cohort by the complete cohort. Since
\[
 a_0Jb^2(2\cdot2^S)^{-2/q}\ge1,
\]
we obtain, via a dyadic frequency decomposition of the sum, that
\begin{align*}
 \sum_{k=1}^H\frac{|\widehat\sigma(k)|^J}{k}
 &\le C_q\sum_{j=0}^S2^{-j}
       \exp\left(-\tfrac12\,2^{2(S-j)/q}\right)\\
 &\le C_q2^{-S}\sum_{h=0}^\infty2^h
       e^{-2^{2h/q}/2}
 \lesssim_q H^{-1},
\end{align*}
where, in the final estimate, we use that the infinite series converges for all $q \geq 1$. By \eqref{eq:ET} and \eqref{eq:discrepancy-cutoff}, the discrepancy is at most $CJ^{-q/2}$. Enlarging $C$ handles the bounded values of $J$ for which the cutoff in \eqref{eq:discrepancy-cutoff} is less than one.

Translating the measure multiplies each Fourier coefficient by a complex number of modulus one. The same estimate therefore holds uniformly in $a$; fix a constant $C_{\mathrm{disc}}\ge1$ for which \eqref{eq:discrepancy-rate} holds for every $a$ and $J$. If $J\ge C_{\mathrm{mix}}\eta^{-2/q}$, with $C_{\mathrm{mix}}=C_{\mathrm{disc}}^{2/q}$, then
\[
 C_{\mathrm{disc}}J^{-q/2}
 \le C_{\mathrm{disc}}C_{\mathrm{mix}}^{-q/2}\eta=\eta.
\]
Thus every interval $I\subset\T$ with $\mu_{\mathbb{T}}(I)=2\eta$ has mass at least $2\eta-C_{\mathrm{disc}}J^{-q/2}\ge\eta$. This proves the interval conclusion for every $J$ above the stated threshold.
\end{proof}

\begin{remark}\label{rem:discrepancy-not-TV}
The estimate \eqref{eq:discrepancy-rate} concerns interval discrepancy. Since $\sigma^{*J}$ is finitely supported for every finite $J$, its total variation distance from Haar measure is maximal.
\end{remark}

\section{Collinear cylinders and the first successful block}\label{sec:stacks}
Throughout this section, $\theta\in[0,\pi)$ is fixed and arbitrary. The covering families may depend on $\theta$; all constants, block lengths, and terminal depths are uniform in $\theta$.

Fix a system satisfying the disk separation assumptions of Definition~\ref{def:disk-separation}. Assume that its alphabet is a disjoint union
\[
 \cA=\cH\sqcup\cB,\qquad
 \#\cH=r\ge2,\quad \#\cB=q\ge1,\quad L=r+q.
\]
We call $\cH$ the \emph{nonrotating part} and $\cB$ the \emph{rotating part} of the alphabet. For each $h\in\cH$, assume that
\begin{equation}\label{eq:collinear-maps}
 F_h(z)=L^{-1}z+c_h,
 \qquad c_h\in\R,\quad z\in\C.
\end{equation}
Enumerate $\cB=\{a_1,\ldots,a_q\}$ and put $\tau_j:=\tau_{a_j}$. We assume that $(\tau_1,\ldots,\tau_q)$ satisfies \eqref{eq:BA} with a fixed constant $b>0$. All constants in this section may depend on this fixed system and on $b$, but are independent of $\theta$ and of the integers specifying the construction depth or block decomposition.

\begin{theorem}\label{thm:structural-projection}
Under these assumptions, define
\begin{equation}\label{eq:structural-exponent}
 A:=\frac{L^{2+2/q}}r>1,
 \qquad \beta:=\frac{\log r}{\log A}
 =\frac{\log_Lr}{2+2/q-\log_Lr}>0.
\end{equation}
There is a constant $C\ge1$ such that, for every integer $N\ge0$,
\begin{equation}\label{eq:structural-projection}
 \sup_{0\le\theta<\pi}\cH^1(\pi_\theta(E_N))
 \le C(N+1)^{-\beta}\bigl[\log(2+N)\bigr]^\beta.
\end{equation}
The same estimate holds with $E_N$ replaced by $\Nbd{L^{-N}}{E}$, after changing $C$. Consequently, for each $\beta_0\in(0,\beta)$, both quantities are bounded by $C_{\beta_0}(N+1)^{-\beta_0}$.
\end{theorem}

For each integer $n\ge1$ and each word $v\in\cH^n$, composing \eqref{eq:collinear-maps} gives
\begin{equation}\label{eq:collinear-word}
 F_v(z)=L^{-n}z+c_v,\qquad c_v\in\R.
\end{equation}
Since $D_v\subset D$, its center satisfies $|c_v|\le1$. Thus all $r^n$ centers at this depth lie on the real axis. After applying any prefix map $F_\omega$, these centers remain collinear, with line direction $2\pi\tau_\omega$.

\begin{lemma}\label{lem:stack-interval}
Let $n\ge1$ be an integer, let $\omega$ be any finite word, possibly empty, and let $\theta\in[0,\pi)$. Suppose that
\begin{equation}\label{eq:favourable-orientation}
 |\cos(2\pi\tau_\omega-\theta)|\le L^{-n}.
\end{equation}
Then
\begin{equation}\label{eq:stack-projection}
 \pi_\theta\left(\bigcup_{v\in\cH^n}D_{\omega v}\right)
 \subset
 \left[\pi_\theta(b_\omega)-2L^{-(|\omega|+n)},\,
       \pi_\theta(b_\omega)+2L^{-(|\omega|+n)}\right].
\end{equation}
The projected union therefore has length at most $4L^{-(|\omega|+n)}$.
\end{lemma}
\begin{proof}
By \eqref{eq:orientation-word} and \eqref{eq:collinear-word},
\[
 F_{\omega v}(z)
 =b_\omega+L^{-|\omega|}\ee{\tau_\omega}c_v
   +L^{-(|\omega|+n)}\ee{\tau_\omega}z.
\]
Since $c_v\in\R$ and $|c_v|\le1$, the projected center differs from $\pi_\theta(b_\omega)$ by at most
\[
 L^{-|\omega|}|c_v|\,|\cos(2\pi\tau_\omega-\theta)|
 \le L^{-(|\omega|+n)}.
\]
The disk $D_{\omega v}$ has radius $L^{-(|\omega|+n)}$, so its projection is contained in the interval in \eqref{eq:stack-projection}. This holds for every $v\in\cH^n$.
\end{proof}

\subsection{A uniform probability of success}
Set
\[
 C_0:=C_{\mathrm{mix}}=C_{\mathrm{disc}}^{2/q}\ge1,
\]
where the constants are those of Proposition~\ref{prop:orientation-discrepancy}. In particular, $C_0$ is independent of the projection direction, the block depth, and the incoming orientation.
For each integer $n\ge1$, define
\begin{equation}\label{eq:block-parameters}
 K_n:=r^n,\qquad \eta_n:=\frac{L^{-n}}{2\pi},\qquad
 J_n:=\left\lceil C_0\eta_n^{-2/q}\right\rceil,
 \qquad s_n:=J_n+n.
\end{equation}
Since $L\ge3$,
\[
 0<\eta_n\le\frac1{6\pi}<\frac1{10}.
\]
The definition of $J_n$ and \eqref{eq:discrepancy-rate} give
\begin{equation}\label{eq:block-discrepancy-bound}
 \Disc(\delta_a*\sigma^{*J_n})
 \le C_{\mathrm{disc}}J_n^{-q/2}
 \le C_{\mathrm{disc}}C_0^{-q/2}\eta_n
 =\eta_n
 \qquad\text{for every }a\in\T.
\end{equation}
For $\theta\in[0,\pi)$, define the closed interval on the circle
\begin{equation}\label{eq:target-orientation-interval}
 I_{n,\theta}:=
 \left\{t\in\T:
 \left\|t-\frac{\theta}{2\pi}-\frac14\right\|_\T\le\eta_n\right\}.
\end{equation}
It has Haar measure $2\eta_n$. If $t\in I_{n,\theta}$, choose a representative $e\in[-\eta_n,\eta_n]$ such that $t=\theta/(2\pi)+1/4+e$ modulo one. Then
\begin{equation}\label{eq:target-implies-favourable}
 |\cos(2\pi t-\theta)|=|\sin(2\pi e)|
 \le2\pi\eta_n=L^{-n}.
\end{equation}
Hence $\tau_\omega\in I_{n,\theta}$ is a sufficient condition for applying Lemma~\ref{lem:stack-interval} to the prefix $\omega$.

Write $\cA^*:=\bigcup_{m=0}^\infty\cA^m$ for the set of finite words, where $\cA^0=\{\varnothing\}$ and $\tau_\varnothing=0$. Concatenation is denoted by juxtaposition.
\begin{definition}\label{def:successful-block}
Fix an integer $n\ge1$ and a direction $\theta\in[0,\pi)$, with the parameters in \eqref{eq:block-parameters}. A \emph{block} is a word $\xi\in\cA^{s_n}$. It has a unique decomposition
\[
 \xi=uv,\qquad u\in\cA^{J_n},\quad v\in\cA^n.
\]
We call $u$ its \emph{orientation part} and $v$ its \emph{collinear part}.

Given a preceding word $\rho\in\cA^*$, possibly empty, we say that $\xi=uv$ is \emph{successful after $\rho$}, at parameters $(n,\theta)$, if both conditions
\begin{equation}\label{eq:successful-block-test}
 \tau_{\rho u}=\tau_\rho+\tau_u\in I_{n,\theta},
 \qquad v\in\cH^n
\end{equation}
hold, with addition of orientations in $\T$.
\end{definition}
If $uv$ is successful after $\rho$, Lemma~\ref{lem:stack-interval} applies with $\omega=\rho u$ and covers the projected disks for all $v'\in\cH^n$ by one interval of length $4L^{-(|\rho|+s_n)}$.

Fix $n$ and $\theta$ through the end of the next subsection; their dependence is suppressed in the event notation. On the product probability space $\Omega=\cA^\N$ from Section~\ref{sec:geometry}, let $X_m$ be the $m$th coordinate symbol and define
\[
 \Theta_0:=0,\qquad
 \Theta_m:=\sum_{i=1}^m\tau_{X_i}\pmod1\qquad\forall m\in\N_{\ge1}.
\]
For the $j$th consecutive block, define
\begin{equation}\label{eq:block-success-event}
 \begin{split}
 A_j&:=\{\Theta_{(j-1)s_n+J_n}\in I_{n,\theta}\}\\
 &\qquad\cap\{X_{(j-1)s_n+J_n+1}\cdots X_{js_n}\in\cH^n\},
 \qquad j\ge1.
 \end{split}
\end{equation}
For each integer $j\ge0$, define
\[
 \mathcal G_j:=
 \left\{\bigcup_{\rho\in S}[\rho]:S\subseteq\cA^{js_n}\right\}.
\]
This is the $\sigma$-algebra of events determined by the first $j$ completed blocks; its atoms are the cylinders $[\rho]$ with $\rho\in\cA^{js_n}$. In particular, $\mathcal G_0=\{\varnothing,\Omega\}$ and $A_j\in\mathcal G_j$ for $j\ge1$. Conditional probabilities with respect to these finite $\sigma$-algebras are taken in the version constant on each of their atoms.

Define the \emph{first-success index} by
\begin{equation}\label{eq:first-success-index}
 T(\boldsymbol\alpha):=\inf\{j\ge1:\boldsymbol\alpha\in A_j\}, \qquad \forall \boldsymbol{\alpha} \in \Omega,
\end{equation}
and we use the convention that $\inf \emptyset = \infty$.

For each integer $j\ge0$, define the \emph{survival event}
\begin{equation}\label{eq:survival-event}
 \mathcal S_j:=\{T>j\}
 =\{\boldsymbol\alpha\in\Omega:T(\boldsymbol\alpha)>j\}
 =\bigcap_{i=1}^j A_i^c.
\end{equation}
Thus $\mathcal S_j$ consists of the addresses for which none of the first $j$ blocks is successful. Complements are taken relative to $\Omega$, and the empty-intersection convention gives $\mathcal S_0=\Omega$. Since $A_i\in\mathcal G_i\subseteq\mathcal G_j$ for $i\le j$, one has $\mathcal S_j\in\mathcal G_j$. In particular, $T$ is a stopping time for $(\mathcal G_j)$; its value $T=j$ corresponds to construction depth $js_n$.

\begin{lemma}\label{lem:success-probability}
Put
\begin{equation}\label{eq:success-probability}
 p_n:=\eta_n\left(\frac rL\right)^n
 =\frac1{2\pi}\frac{r^n}{L^{2n}}\in(0,1).
\end{equation}
For every integer $j\ge1$,
\begin{equation}\label{eq:conditional-block-success}
 \Prob(A_j\mid\mathcal G_{j-1})
 =\left(\frac rL\right)^n
   (\delta_{\Theta_{(j-1)s_n}}*\sigma^{*J_n})(I_{n,\theta})
 \ge p_n.
\end{equation}
Consequently, for every integer $B\ge0$, one has
\begin{equation}\label{eq:failure-probability}
 \Prob(\mathcal S_B)\le(1-p_n)^B\le e^{-p_nB}.
\end{equation}
All these bounds are uniform in $\theta$ and in the preceding word.
\end{lemma}
\begin{proof}
Fix $j\ge1$ and an atom $[\rho] \in \mathcal G_{j-1}$ with $\rho\in\cA^{(j-1)s_n}$. Conditional on $[\rho]$, the orientation after the next $J_n$ coordinates has distribution $\delta_{\tau_\rho}*\sigma^{*J_n}$. The last $n$ coordinates are independent of both the prefix and these $J_n$ coordinates, and belong to $\cH^n$ with probability $(r/L)^n$. This proves the equality in \eqref{eq:conditional-block-success} on every atom. Since $\mu_{\mathbb T}(I_{n,\theta})=2\eta_n$, \eqref{eq:block-discrepancy-bound} implies
\[
 (\delta_{\tau_\rho}*\sigma^{*J_n})(I_{n,\theta})
 \ge2\eta_n-\eta_n=\eta_n.
\]
Multiplication by $(r/L)^n$ gives \eqref{eq:conditional-block-success}.

For every integer $j\ge1$, \eqref{eq:survival-event} gives
\[
 \mathcal S_{j-1}=\bigcap_{i=1}^{j-1}A_i^c,
 \qquad
 \mathcal S_j=\mathcal S_{j-1}\cap A_j^c.
\]
For $1\le i\le j-1$, one has $A_i\in\mathcal G_i\subseteq\mathcal G_{j-1}$. Closure under complements and finite intersections therefore gives $\mathcal S_{j-1}\in\mathcal G_{j-1}$. For $j=1$, the empty intersection is $\Omega$, which belongs to $\mathcal G_0$.

The tower property of conditional expectation and the $\mathcal G_{j-1}$-measurability of $\one_{\mathcal S_{j-1}}$ imply
\begin{align*}
 \Prob(\mathcal S_j)
 &=\E\left[\one_{\mathcal S_{j-1}}\one_{A_j^c}\right]\\
 &=\E\left[\E\left[
       \one_{\mathcal S_{j-1}}\one_{A_j^c}
       \mid\mathcal G_{j-1}\right]\right]\\
 &=\E\left[\one_{\mathcal S_{j-1}}
       \E\left[\one_{A_j^c}\mid\mathcal G_{j-1}\right]\right]\\
 &=\E\left[\one_{\mathcal S_{j-1}}
       \Prob(A_j^c\mid\mathcal G_{j-1})\right].
\end{align*}
By \eqref{eq:conditional-block-success},
\[
 \Prob(A_j^c\mid\mathcal G_{j-1})
 =1-\Prob(A_j\mid\mathcal G_{j-1})\le1-p_n.
\]
Multiplying by the nonnegative indicator $\one_{\mathcal S_{j-1}}$ and taking expectations yields
\[
 \Prob(\mathcal S_j)
 \le(1-p_n)\E[\one_{\mathcal S_{j-1}}]
 =(1-p_n)\Prob(\mathcal S_{j-1}).
\]
Since $\mathcal S_0=\Omega$, one has $\Prob(\mathcal S_0)=1$. Iterating this inequality for $j=1,\ldots,B$ gives
\[
 \Prob(\mathcal S_B)\le(1-p_n)^B\Prob(\mathcal S_0)=(1-p_n)^B
 \qquad\forall B\in\N_{\ge1}.
\]
The same bound holds with equality for $B=0$. Finally, $1-p_n\le e^{-p_n}$ implies $(1-p_n)^B\le e^{-p_nB}$ for every integer $B\ge0$, proving \eqref{eq:failure-probability}.
\end{proof}

\subsection{The first-success cover}
\begin{proposition}\label{prop:first-success-cover}
For every integer $n\ge1$, direction $\theta\in[0,\pi)$, and integer $B\ge1$, let $M:=Bs_n$, with the parameters in \eqref{eq:block-parameters}. Then
\begin{equation}\label{eq:first-success-cover}
 \cH^1(\pi_\theta(E_M))\le\frac4{K_n}+2e^{-p_nB}.
\end{equation}
The parameters $K_n,s_n,p_n$, and therefore the depth $M$ for a prescribed $B$, are independent of $\theta$.
\end{proposition}
\begin{proof}
Fix $n,\theta,B$ as in the statement. For $0\le j\le B$, define the residual family
\begin{equation}\label{eq:residual-word-family}
 \mathcal R_j:=
 \{\rho\in\cA^{js_n}:[\rho]\subset\mathcal S_j\}.
\end{equation}
Since $\mathcal S_j\in\mathcal G_j$, the residual family gives the exact representation
\begin{equation}\label{eq:residual-word-mass}
 \begin{aligned}
 \mathcal S_j&=\bigsqcup_{\rho\in\mathcal R_j}[\rho],\\
 \mathcal R_0&=\{\varnothing\},\qquad
 \Prob(\mathcal S_j)=\#\mathcal R_j\,L^{-js_n}.
 \end{aligned}
\end{equation}

For $1\le j\le B$, define the family of orientation prefixes
\begin{equation}\label{eq:orientation-prefix-family}
 \mathcal O_j:=
 \{\rho u:\rho\in\mathcal R_{j-1},\ u\in\cA^{J_n},\
             \tau_{\rho u}\in I_{n,\theta}\}.
\end{equation}
Each $\omega\in\mathcal O_j$ has length $(j-1)s_n+J_n$. For every $v\in\cH^n$, the word $\omega v$ has first success at index $j$, since the choice of $v$ affects neither the preceding blocks nor the orientation test. Define
\begin{equation}\label{eq:stopped-word-family}
 \mathcal W_j:=\{\omega v:\omega\in\mathcal O_j,\ v\in\cH^n\}.
\end{equation}
The prescribed component lengths make $(\omega,v)\mapsto\omega v$ injective. Thus each $\omega\in\mathcal O_j$ indexes exactly $K_n=r^n$ stopped words.

Every terminal word either belongs to $\mathcal R_B$ or has a unique first-success index $j\le B$ and a prefix in $\mathcal W_j$. Every terminal extension of a word in $\mathcal W_j$ has the same first-success index. Hence
\begin{equation}\label{eq:first-success-word-partition}
 \cA^M
 =\mathcal R_B\ \sqcup\ \bigg(
   \bigsqcup_{j=1}^B\ \bigsqcup_{\xi\in\mathcal W_j}
   \{\xi t:t\in\cA^{M-js_n}\} \bigg).
\end{equation}
Here $\cA^0=\{\varnothing\}$ covers the case $j=B$.

In particular, $\bigcup_{j=1}^B\mathcal W_j$ is prefix-free. Indeed, if a word stopped at index $j$ were a prefix of one stopped at a later index $k$, the latter would already have succeeded at index $j$. This contradicts its first success being at $k$. Distinct stopped words at the same index have equal length. The families $\mathcal O_j$ of orientation prefixes are not required to have this property.

A stopped word in $\mathcal W_j$ has exactly $L^{M-js_n}$ terminal extensions. Counting the disjoint pieces in \eqref{eq:first-success-word-partition} and dividing by $L^M$ gives
\begin{equation}\label{eq:successful-group-mass}
 \sum_{j=1}^B\sum_{\omega\in\mathcal O_j}K_nL^{-js_n}
 =1-\#\mathcal R_B L^{-M}
 =\Prob(\mathcal S_B^c)\le1.
\end{equation}
In particular, this recovers Lemma~\ref{lem:prefix-free} for the stopped family.

For $\omega\in\mathcal O_j$, we have $\tau_\omega\in I_{n,\theta}$ and $|\omega|+n=js_n$. Lemma~\ref{lem:stack-interval} and \eqref{eq:target-implies-favourable} therefore cover
\[
 \pi_\theta\left(\bigcup_{v\in\cH^n}D_{\omega v}\right)
\]
by one interval of length at most $4L^{-js_n}$. Every terminal descendant of $\omega v$ is contained in $D_{\omega v}$, so the same interval covers all terminal descendants of that group. Each residual disk $D_\rho$, with $\rho\in\mathcal R_B$, has projected length $2L^{-M}$. These intervals cover $\pi_\theta(E_M)$ by \eqref{eq:first-success-word-partition}. Subadditivity of length, \eqref{eq:successful-group-mass}, and \eqref{eq:residual-word-mass} yield
\begin{align}
 \cH^1(\pi_\theta(E_M))
 &\le4\sum_{j=1}^B\sum_{\omega\in\mathcal O_j}L^{-js_n}
       +2\#\mathcal R_B L^{-M}\notag\\
 &=\frac4{K_n}\Prob(\mathcal S_B^c)+2\Prob(\mathcal S_B)\notag\\
 &\le\frac4{K_n}+2e^{-p_nB}.
 \label{eq:first-success-cover-cost}
\end{align}
The last inequality is \eqref{eq:failure-probability}.
\end{proof}

\begin{remark}
Prefix freeness is required only for the stopped words. The covering intervals may overlap; their lengths are estimated by subadditivity.
\end{remark}

\subsection{Depth selection and interpolation}
\begin{proof}[Proof of Theorem~\ref{thm:structural-projection}]
For each integer $n\ge1$, choose
\begin{equation}\label{eq:number-of-blocks}
 B_n:=\left\lceil\frac{n\log r}{p_n}\right\rceil,
 \qquad M_n:=B_ns_n=B_n(J_n+n).
\end{equation}
Since $p_nB_n\ge n\log r$, we have $e^{-p_nB_n}\le r^{-n}$. Apply Proposition~\ref{prop:first-success-cover} with $B=B_n$ and $K_n=r^n$. Its depth $M_n$ is independent of $\theta$, and it gives
\begin{equation}\label{eq:projection-at-special-depths}
 \sup_{0\le\theta<\pi}\cH^1(\pi_\theta(E_{M_n}))\le6r^{-n}.
\end{equation}

Substitution of \eqref{eq:block-parameters} and \eqref{eq:success-probability} gives
\begin{align*}
 J_n&=\left\lceil C_0(2\pi)^{2/q}(L^{2/q})^n\right\rceil,\\
 B_n&=\left\lceil2\pi(\log r)n(L^2/r)^n\right\rceil.
\end{align*}
Both bases $L^{2/q}$ and $L^2/r$ exceed one. The ceilings change these positive expressions by at most fixed multiplicative factors, and hence
\[
 J_n\asymp L^{2n/q},\qquad
 B_n\asymp n(L^2/r)^n.
\]
For $\lambda:=L^{2/q}>1$, the geometric sum gives
\[
 n\le\frac{\lambda^n-1}{\lambda-1}
 \le\frac{\lambda^n}{\lambda-1}.
\]
Consequently $s_n=J_n+n\asymp L^{2n/q}$. There are constants $0<c_1\le c_2$, depending only on the fixed system, such that
\begin{equation}\label{eq:terminal-depth-cost}
 c_1nA^n\le M_n\le c_2nA^n
 \qquad\forall n\in\N_{\ge1},\qquad A=\frac{L^{2+2/q}}r.
\end{equation}

The explicit formulas show that $J_n$ and $B_n$ are nondecreasing, while $s_n=J_n+n$ is strictly increasing. Since $B_n\ge1$, the sequence $M_n=B_ns_n$ is strictly increasing. It tends to infinity by \eqref{eq:terminal-depth-cost}, and its successive ratios satisfy
\[
 \frac{M_{n+1}}{M_n}
 \le\frac{c_2}{c_1}\frac{n+1}{n}A
 \le\frac{2c_2A}{c_1}.
\]

For each integer $N\ge M_1$, let $n\ge1$ be the unique index satisfying $M_n\le N<M_{n+1}$. Nestedness gives
\begin{equation}\label{eq:interpolation-projection}
 \sup_{0\le\theta<\pi}\cH^1(\pi_\theta(E_N))
 \le\sup_{0\le\theta<\pi}\cH^1(\pi_\theta(E_{M_n}))
 \le6r^{-n}.
\end{equation}
The two sides of \eqref{eq:terminal-depth-cost}, applied at $n$ and $n+1$, respectively, give
\[
 c_1A^n\le M_n\le N,
 \qquad
 N<M_{n+1}\le c_2(n+1)A^{n+1}\le2c_2AnA^n.
\]
The first inequality implies $n\lesssim\log(2+N)$; the second implies $A^n>N/(2c_2An)$. Since $r=A^\beta$, it follows that
\[
 r^{-n}=(A^n)^{-\beta}
 \le(2c_2A)^\beta N^{-\beta}n^\beta
 \lesssim N^{-\beta}[\log(2+N)]^\beta.
\]
Together with \eqref{eq:interpolation-projection}, this proves \eqref{eq:structural-projection} for $N\ge M_1$. For the finitely many integers $0\le N<M_1$, use $E_N\subset D$ and $\cH^1(\pi_\theta(E_N))\le2$, and enlarge the constant. Replacing $N$ by $N+1$ in the power factor changes only the constant for $N\ge1$.

For the neighbourhood assertion, retain $R<1$ from \eqref{eq:strict-interior} and fix the integer
\[
 c_*:=\left\lceil\log_L\bigl((1-R)^{-1}\bigr)\right\rceil.
\]
Then $(1-R)L^{c_*}\ge1$. For every $N\ge c_*$, Lemma~\ref{lem:neighbourhood-comparison} gives
\[
 \Nbd{L^{-N}}{E}\subset E_{N-c_*},
\]
since $L^{-N}\le(1-R)L^{-(N-c_*)}$. Applying the iterate estimate at depth $N-c_*$ gives the claimed bound for all sufficiently large $N$; the fixed shift changes only the constant. The remaining finitely many depths are handled using $\Nbd{L^{-N}}{E}\subset\Nbd{1}{D}$, whose projection has length at most $4$.

Finally, for each $\beta_0\in(0,\beta)$, the function
\[
 (N+1)^{-(\beta-\beta_0)}[\log(2+N)]^\beta
\]
is bounded for $N\ge0$. This absorbs the logarithmic factor and proves the final power-law assertion.
\end{proof}

\section{Algebraic rotations and explicit systems}\label{sec:explicit}
For every integer $q\ge1$, the algebraic vector constructed below satisfies \eqref{eq:BA} with an explicit constant. This supplies the rotation condition for both the sixteen-map example and the family realizing every exponent in $(0,1)$.

\subsection{Algebraic rotation parameters}
\begin{lemma}\label{lem:algebraic-BA}
Let $q\ge1$ be an integer and $\xi_q:=2^{1/(q+1)}$. Then
\begin{equation}\label{eq:algebraic-BA}
 \max_{1\le j\le q}\|k\xi_q^j\|_\T
 \ge b_q k^{-1/q}\qquad \forall k\in\N_{\ge1},
\end{equation}
where one may take
\begin{equation}\label{eq:explicit-BA-constant}
 b_q:=\left(\frac52\,4^q(q+1)^{(q+1)/2}\right)^{-1/q}.
\end{equation}
For $q=3$, the smaller constant $b_3=1/16$ is also valid.
\end{lemma}
\begin{proof}
Put $m=q+1$ and let $A$ be the integer matrix defined by
\[
 A(x_0,x_1,\ldots,x_q)^T=(x_1,x_2,\ldots,x_q,2x_0)^T.
\]
Its characteristic polynomial is $X^m-2$, which is irreducible over $\Q$ by Eisenstein's criterion. For each $k\ge1$, choose nearest integers $p_j$ to $k\xi_q^j$ and write
\[
 v=(k,p_1,\ldots,p_q)^T,
 \quad u=(1,\xi_q,\ldots,\xi_q^q)^T,
 \quad \varepsilon:=\max_j|p_j-k\xi_q^j|.
\]
The determinant $\det[v,Av,\ldots,A^qv]$ is a nonzero integer. Indeed, if it vanished, there would be a nonzero polynomial $P\in\Q[X]$ of degree at most $q$ with $P(A)v=0$. Since $P$ is relatively prime to $X^m-2$, B\'ezout's identity would then make $P(A)$ invertible, forcing $v=0$. Its first coordinate is $k\ge1$, a contradiction.

Now $Au=\xi_qu$, and $v=ku+e$ with $\|e\|_\infty\le\varepsilon\le1/2$. Since $\xi_q^j<2$ for every $1 \leq j\le q$, one has
\[
 \|v\|_2\le\tfrac52\sqrt m\,k.
\]
For $1\le j\le q$, every row of $A^j$ has one nonzero entry, equal to either $1$ or $2$, so $\|A^j\|_{\infty\to\infty}\le2$. Consequently,
\[
 \|A^jv-\xi_q^jv\|_2
 =\|A^je-\xi_q^je\|_2
 \le4\sqrt m\,\varepsilon.
\]
Subtract $\xi_q^j$ times the first column from the $(j+1)$st column in the determinant. Hadamard's inequality gives
\[
 1\le|\det[v,Av,\ldots,A^qv]|
 \le\frac52\,4^q m^{(q+1)/2}k\varepsilon^q.
\]
Rearranging proves \eqref{eq:algebraic-BA}--\eqref{eq:explicit-BA-constant}. When $q=3$, the coefficient on the right is $2560<16^3$, which proves the final assertion.
\end{proof}

Algebraic constructions of badly approximable vectors are classical; see \cite[Section~3]{HS2004}. Lemma~\ref{lem:algebraic-BA} specifies both the vector and a valid Diophantine constant.

\subsection{The sixteen-map system}
Let $E$ be the attractor of the maps in \eqref{eq:intro-sixteen-maps}, and let $E_N$ be its disk iterates. Each first-level image disk has radius $1/16$. The thirteen real centers have mutual distances at least $1/7>1/8$ and modulus at most $6/7$. The remaining centers have modulus at most $1/\sqrt2$, mutual distances at least $1/2$, and distance at least $1/2$ from every real center. Thus the distance between any two first-level disks is at least $1/7-1/8>0$.

A separation constant and an interior radius in Definition~\ref{def:disk-separation} are therefore
\begin{equation}\label{eq:sixteen-geometry-constants}
 g=\frac17-\frac18=\frac1{56},\qquad
 R=\frac67+\frac1{16}=\frac{103}{112}<1.
\end{equation}
Indeed, $1/\sqrt2+1/16<R$, so every first-level disk lies in $\overline{B(0,R)}$. Proposition~\ref{prop:regularity} applies, and the rotation vector $(\xi,\xi^2,\xi^3)$ satisfies \eqref{eq:BA} with $b=1/16$ by Lemma~\ref{lem:algebraic-BA}.

The collinear row in \eqref{eq:intro-sixteen-maps} concerns the cylinder centers. Each $F_\omega(E)$ is a similar copy of the entire attractor. In particular, if
\[
 \Lambda_\omega:=b_\omega+\ee{\tau_\omega}\R
\]
is the line containing its row centers, then $E\subset D$ and the positions of the three rotating disks give
\begin{equation}\label{eq:sixteen-off-row-descendants}
 \dist\bigl(F_\omega(T_\ell(E)),\Lambda_\omega\bigr)
 \ge\left(\frac12-\frac1{16}\right)16^{-|\omega|}
 =\frac7{16}\,16^{-|\omega|}
 \qquad\forall\ell\in\{1,2,3\}.
\end{equation}
Thus off-row descendants occur at a fixed relative distance inside every cylinder. A nonrotating symbol preserves the accumulated row direction, while a symbol $T_\ell$ adds the angle $2\pi\xi^\ell$. The $13^n$ words using only nonrotating symbols have total symbolic mass $(13/16)^n\to0$.

Rotation leaves a circular disk unchanged; its effect is visible in the positions of the descendants. Figure~\ref{fig:sixteen-disks} displays $E_1,E_2$ and magnified configurations inside $D_{T_1}:=T_1(D)$. For every $N\ge1$, disk separation gives the exact identity
\begin{equation}\label{eq:sixteen-magnified-iterate}
 \mathcal Z_N:=16\bigl((E_N\cap D_{T_1})-v_1\bigr)
 =\ee{\xi}E_{N-1}.
\end{equation}
This magnification translates and rescales the configuration without changing the coordinate directions.

\begin{figure}[t]
\centering
\begingroup
\definecolor{sixteenblue}{RGB}{31,93,153}
\newcommand{\SixteenFirstStage}{%
 \foreach \k in {0,...,12}{%
  \fill[black] ({(\k-6)/7},0) circle[radius=0.0625];%
 }%
 \foreach \cx in {-0.5,0,0.5}{%
  \fill[sixteenblue] (\cx,0.5) circle[radius=0.0625];%
 }%
}
\newcommand{\SixteenSecondStage}{%
 \foreach \h in {0,...,12}{%
  \pgfmathsetmacro{\parentx}{(\h-6)/7}%
  \begin{scope}[shift={(\parentx,0)},scale=0.0625]
   \SixteenFirstStage
  \end{scope}%
 }%
 \foreach \parentx/\rotangle in
 {-0.5/68.11456140098,0/149.11688245431,0.5/245.44541898267}{%
  \begin{scope}[shift={(\parentx,0.5)},rotate=\rotangle,scale=0.0625]
   \SixteenFirstStage
  \end{scope}%
 }%
}
\newcommand{\SixteenParentCircle}{%
 \draw[gray!45,line width=0.25pt] (0,0) circle[radius=1];%
}
\begin{tikzpicture}[x=1cm,y=1cm]
 \begin{scope}[scale=2]
  \SixteenParentCircle
  \SixteenFirstStage
  \node[font=\small,anchor=south] at (0,1.12) {(a) $E_1$};
  \node[font=\scriptsize,anchor=south,text=sixteenblue] at (-0.5,0.59) {$T_1$};
  \node[font=\scriptsize,anchor=south,text=sixteenblue] at (0,0.59) {$T_2$};
  \node[font=\scriptsize,anchor=south,text=sixteenblue] at (0.5,0.59) {$T_3$};
 \end{scope}
 \begin{scope}[shift={(5.2,0)},scale=2]
  \SixteenParentCircle
  \SixteenSecondStage
  \draw[sixteenblue,densely dashed,line width=0.3pt]
    (-0.5,0.5) circle[radius=0.0625];
  \node[font=\scriptsize,anchor=east,text=sixteenblue] at (-0.61,0.5) {$T_1(D)$};
  \node[font=\small,anchor=south] at (0,1.12) {(b) $E_2$};
 \end{scope}
 \begin{scope}[shift={(0,-4.8)},scale=2]
  \SixteenParentCircle
  \begin{scope}[rotate=68.11456140098]
   \SixteenFirstStage
  \end{scope}
  \node[font=\small,anchor=south] at (0,1.12) {(c) $\mathcal Z_2$};
  \node[font=\footnotesize,anchor=north] at (0,-1.09) {Magnification $\times16$; axes fixed};
 \end{scope}
 \begin{scope}[shift={(5.2,-4.8)},scale=2]
  \SixteenParentCircle
  \begin{scope}[rotate=68.11456140098]
   \SixteenSecondStage
  \end{scope}
  \node[font=\small,anchor=south] at (0,1.12) {(d) $\mathcal Z_3$};
  \node[font=\footnotesize,anchor=north] at (0,-1.09) {Magnification $\times16$; axes fixed};
 \end{scope}
\end{tikzpicture}
\endgroup
\caption{The first two disk iterates and magnified portions of the second and third iterates. The lower panels show $\mathcal Z_2$ and $\mathcal Z_3$ from \eqref{eq:sixteen-magnified-iterate}, so the rotation of the descendants is retained. The dashed circle in (b) marks the parent $T_1(D)$; the gray circles mark the containing disks. Black terminal disks have a nonrotating final symbol, and blue disks have a rotating final symbol. All disk radii are drawn to scale.}
\label{fig:sixteen-disks}
\end{figure}
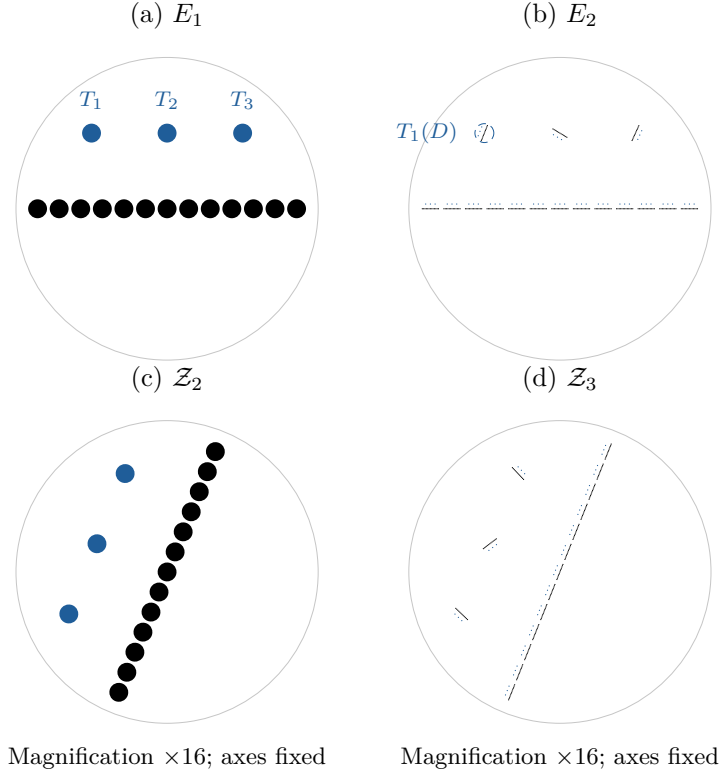

\begin{corollary}\label{cor:sixteen-endpoint}
Let $E$ be the attractor of \eqref{eq:intro-sixteen-maps}, and let $E_N$ be its disk iterates. This system has $L=16$ symbols, of which $r=13$ are nonrotating and $q=3$ are rotating, and its rotation vector satisfies \eqref{eq:BA} with $b=1/16$. Set
\[
 A_*:=\frac{16^{8/3}}{13},\qquad
 \beta_*:=\frac{\log13}{\log A_*}
 =\frac{\log13}{(8/3)\log16-\log13}.
\]
There is a constant $C\ge1$, depending only on these sixteen similarities, such that, for every integer $N\ge0$,
\begin{equation}\label{eq:sixteen-endpoint}
 \sup_{0\le\theta<\pi}\cH^1(\pi_\theta(E_N))
 +\sup_{0\le\theta<\pi}\cH^1\bigl(\pi_\theta(\Nbd{16^{-N}}{E})\bigr)
 \le C(N+1)^{-\beta_*}[\log(2+N)]^{\beta_*}.
\end{equation}
Consequently, for every $\beta_0\in(0,\beta_*)$, there is a constant $C_{\beta_0}$ such that the left-hand side is at most $C_{\beta_0}(N+1)^{-\beta_0}$ for every $N\ge0$. In particular, $\beta_0=53/100$ is admissible. All constants are independent of the projection direction and construction depth.
\end{corollary}
\begin{proof}
The disk separation conditions were verified in \eqref{eq:sixteen-geometry-constants}, and Lemma~\ref{lem:algebraic-BA} supplies the stated Diophantine constant for $(\xi,\xi^2,\xi^3)$. Thus Theorem~\ref{thm:structural-projection} applies with $L=16$, $r=13$, and $q=3$. Its parameter $A=L^{2+2/q}/r$ is $A_*$, and its exponent $\beta=\log r/\log A$ is precisely $\beta_*$.

More explicitly, fix the discrepancy constant $C_{\mathrm{disc}}$ for this system and put $C_0=C_{\mathrm{disc}}^{2/3}$. For each integer $n\ge1$, the parameters in Section~\ref{sec:stacks} become
\[
 \begin{aligned}
 K_n&=13^n,\qquad \eta_n=\frac{16^{-n}}{2\pi},\\
 J_n&=\left\lceil C_0(2\pi)^{2/3}16^{2n/3}\right\rceil,\\
 p_n&=\frac1{2\pi}\frac{13^n}{16^{2n}},\\
 B_n&=\left\lceil2\pi n(\log13)\left(\frac{16^2}{13}\right)^n\right\rceil,\\
 M_n&=B_n(J_n+n).
 \end{aligned}
\]
Here $p_n$ is the uniform lower bound for conditional success, and $B_n$ is chosen so that $e^{-p_nB_n}\le13^{-n}$. Proposition~\ref{prop:first-success-cover} therefore gives
\[
 \sup_{0\le\theta<\pi}\cH^1(\pi_\theta(E_{M_n}))\le6\,13^{-n},
 \qquad M_n\asymp nA_*^n.
\]
The all-depth interpolation in the proof of Theorem~\ref{thm:structural-projection} now gives
\[
 \sup_{0\le\theta<\pi}\cH^1(\pi_\theta(E_N))
 \lesssim(N+1)^{-\beta_*}[\log(2+N)]^{\beta_*}
 \qquad\forall N\in\N_{\ge0}.
\]
Indeed, for $M_n\le N<M_{n+1}$, nestedness bounds the projection by $6\,13^{-n}$, while $M_n\asymp nA_*^n$ implies $n\lesssim\log(2+N)$ and $A_*^n\gtrsim N/n$. The identity $13=A_*^{\beta_*}$ yields the displayed estimate; the finitely many depths below $M_1$ are absorbed into the constant.

For the neighbourhood term, the interior radius $R=103/112$ satisfies $16(1-R)=9/7>1$. Lemma~\ref{lem:neighbourhood-comparison} consequently gives
\[
 \Nbd{16^{-N}}{E}\subset E_{N-1}\qquad\forall N\in\N_{\ge1}.
\]
The iterate estimate at depth $N-1$ has the same right-hand side up to a fixed factor. The depth $N=0$ is bounded using $\Nbd{1}{E}\subset\Nbd{1}{D}$. This proves \eqref{eq:sixteen-endpoint}.

For every $\beta_0<\beta_*$, the function
\[
 (N+1)^{-(\beta_*-\beta_0)}[\log(2+N)]^{\beta_*}
\]
is bounded for $N\ge0$, so the logarithmic factor can be absorbed into the constant in the weaker power bound. Finally, $\beta_*>53/100$ is equivalent to
\[
 13^{459}>16^{424}.
\]
A short integer certificate is
\[
 13^{10}=137858491849>137438953472=2\cdot16^9.
\]
Raising this inequality to the power $46$ and dividing by $13$ gives
\[
 13^{459}>\frac{2^{46}16^{414}}{13}
 =\frac{64}{13}\,16^{424}>16^{424},
\]
as required.
\end{proof}

\subsection{A family with an arbitrary prescribed exponent}
We now vary the total number $L$ of similarities and the number $q$ of rotating similarities, subject to the separation conditions below.

Fix integers $q\ge2$ and $L$ satisfying
\begin{equation}\label{eq:general-packing-restrictions}
 L\ge q+2,\qquad L>2(q-1),\qquad L>4.
\end{equation}
Put $r=L-q$, $\xi_q=2^{1/(q+1)}$, and define
\begin{equation}\label{eq:general-explicit-maps}
 \begin{split}
 H_j(z)&:=L^{-1}z+\frac{2j-r+1}{r},\qquad \forall j\in\{0,\ldots,r-1\},\\
 T_\ell(z)&:=L^{-1}\ee{\xi_q^\ell}z
       -\frac12+\frac{\ell-1}{q-1}+\frac i2,
       \qquad \forall \ell\in\{1,\ldots,q\}.
 \end{split}
\end{equation}
There are exactly $r+q=L$ maps, each of contraction factor $L^{-1}$. The row centers have spacing $2/r$ and maximum modulus $1-1/r$. The rotating centers have spacing $1/(q-1)$, modulus at most $1/\sqrt2$, and distance at least $1/2$ from the row centers. Hence Definition~\ref{def:disk-separation} holds with
\[
 \begin{aligned}
 g&\ge\min\left\{\frac2r-\frac2L,\frac1{q-1}-\frac2L,\frac12-\frac2L\right\}>0,\\
 R&=\max\left\{1-\frac1r+\frac1L,\frac1{\sqrt2}+\frac1L\right\}<1.
 \end{aligned}
\]
The strict inequalities follow from \eqref{eq:general-packing-restrictions} and $r=L-q<L$. Lemma~\ref{lem:algebraic-BA} gives \eqref{eq:BA} with $\tau=(\xi_q,\ldots,\xi_q^q)$ and $b=b_q$.

The exponent in Theorem~\ref{thm:structural-projection} is
\begin{equation}\label{eq:family-exponent}
 \beta(q,L):=\frac{\log(L-q)}{(2+2/q)\log L-\log(L-q)}.
\end{equation}
For fixed $q$,
\begin{equation}\label{eq:limit-fixed-q}
 \lim_{L\to\infty}\beta(q,L)=\frac q{q+2}.
\end{equation}
Thus every positive exponent below $q/(q+2)$ is obtained by fixing $q$ and taking $L$ sufficiently large. In particular, $q=3$ gives every positive exponent below $3/5$, and the union of these ranges over $q\ge2$ is $(0,1)$.

For an explicit choice for each prescribed $\alpha\in(0,1)$, fix an integer
\begin{equation}\label{eq:q-choice}
 q>\frac{2\alpha}{1-\alpha},\qquad q\ge2,
\end{equation}
and set
\[
 t:=\frac{\alpha(2+2/q)}{1+\alpha}<1.
\]
Choose any integer
\begin{equation}\label{eq:L-choice}
 L\ge\left\lceil\max\left\{2q+2,5,\frac{4q}{1-t}\right\}\right\rceil.
\end{equation}
For $0\le x\le1/2$, one has $\log(1-x)\ge-2x$. Since \eqref{eq:L-choice} gives $q/L\le1/2$ and $2q/L\le(1-t)/2$,
\[
 \log_L(L-q)=1+\frac{\log(1-q/L)}{\log L}
 \ge1-\frac{2q}{L\log L}
 \ge1-\frac{1-t}{2\log L}>t,
\]
where the strict inequality uses $L\ge5$. Writing $x=\log_L(L-q)$, the denominator $2+2/q-x$ is positive, and
\[
 \beta(q,L)=\frac{x}{2+2/q-x}>\alpha
 \quad\Longleftrightarrow\quad
 x>\frac{\alpha(2+2/q)}{1+\alpha}=t.
\]
The choices \eqref{eq:q-choice}--\eqref{eq:L-choice} also imply \eqref{eq:general-packing-restrictions}. Theorem~\ref{thm:structural-projection}, with the logarithmic factor absorbed using $\beta(q,L)>\alpha$, therefore gives \eqref{eq:arbitrary-projection}. The resulting system is fixed once $\alpha$, $q$, and $L$ are fixed. The estimate is uniform in depth and direction; its constant may depend on these choices.

\section{The analytic capacity of the iterates}\label{sec:capacity}
Throughout this section the system is fixed, and implicit constants may depend on its similarities and separation constants. The capacity estimate requires disk separation, noncollinearity, and the critical relation between branching and contraction.

\begin{theorem}\label{thm:general-capacity}
Let $\cF$ be a system of exactly $L$ similarities of contraction factor $L^{-1}$ satisfying Definition~\ref{def:disk-separation}, and suppose that its attractor $E$ is not contained in an affine line. Then, for every integer $N\ge0$,
\begin{equation}\label{eq:general-capacity}
 \gamma(E_N)\asymp(N+1)^{-1/2}.
\end{equation}
Consequently, for all sufficiently small $\delta>0$,
\begin{equation}\label{eq:neighbourhood-capacity}
 \gamma(\Nbd{\delta}{E})\asymp(\log\delta^{-1})^{-1/2}.
\end{equation}
The constants depend only on the fixed system.
\end{theorem}

\subsection{The curvature characterization}
For three distinct points $x,y,z\in\C$, their \emph{Menger curvature} \cite{Tolsa2003,Tolsa2005} is
\[
 c(x,y,z):=\frac1{R(x,y,z)},
\]
where $R(x,y,z)$ is the radius of their circumcircle. We set $c=0$ for a collinear triple or when two points coincide. For a finite positive Borel measure $\nu$, write
\begin{equation}\label{eq:curvature-energy}
 c^2(\nu):=\iiint c(x,y,z)^2\,d\nu(x)\,d\nu(y)\,d\nu(z).
\end{equation}
We use the following characterization, due to Tolsa \cite{Tolsa2003}; the stated form is also recorded in \cite[Theorem~A]{Tolsa2005}.

\begin{theorem}\label{thm:Tolsa}
For every compact $K\subset\C$,
\begin{equation}\label{eq:Tolsa-characterization}
 \gamma(K)\asymp\sup_\nu\nu(K),
\end{equation}
where the supremum is over finite positive Borel measures supported on $K$ such that
\begin{equation}\label{eq:admissible-capacity-measures}
 \begin{aligned}
 \nu(B(z,t))&\le t
 \qquad \forall z\in\C,\ \forall t>0,\\
 c^2(\nu)&\le\nu(K).
 \end{aligned}
\end{equation}
The comparison constants in \eqref{eq:Tolsa-characterization} are absolute.
\end{theorem}

The formulation with growth imposed only for centers in the support is equivalent up to absolute constants. Indeed, a ball meeting the support is contained in a ball of twice its radius centered on the support; replacing the measure by one half of itself gives the stated growth bound and preserves the curvature inequality.

\subsection{The lower capacity bound}
Put $\varepsilon=L^{-N}$ and let $\mu_N$ be the probability measure assigning mass $L^{-N}$ uniformly by area to each depth-$N$ disk. Since those disks are disjoint,
\begin{equation}\label{eq:area-measure}
 d\mu_N(z)=\frac1{\pi\varepsilon}\one_{E_N}(z)\,d\cH^2(z).
\end{equation}
Its mass in every depth-$k$ cylinder, $k\le N$, is $L^{-k}$.

\begin{lemma}\label{lem:area-growth}
Uniformly in $N$, $z\in\C$, and $t>0$,
\begin{equation}\label{eq:area-growth}
 \mu_N(B(z,t))\le C\min\{t,t^2/\varepsilon\}.
\end{equation}
\end{lemma}
\begin{proof}
The density in \eqref{eq:area-measure} gives $\mu_N(B(z,t))\le t^2/\varepsilon$. For $t\ge\varepsilon$, let $\mathcal W$ be the set of depth-$N$ words whose disks meet $B(z,t)$. Every such disk is contained in $B(z,t+2\varepsilon)$, and $\mu(D_w)=\mu_N(D_w)=L^{-N}$ for $w\in\mathcal W$. Hence
\[
 \mu_N(B(z,t))\le\sum_{w\in\mathcal W}\mu(D_w)
 \le\mu(B(z,t+2\varepsilon))\lesssim t,
\]
where Proposition~\ref{prop:regularity} supplies the last inequality. For $t<\varepsilon$, the quadratic estimate implies the linear estimate as well. For $t$ exceeding the diameter of the root disk, use $\mu_N(\C)=1$.
\end{proof}

\begin{lemma}\label{lem:natural-curvature-upper}
One has
\begin{equation}\label{eq:natural-curvature-upper}
 c^2(\mu_N)\lesssim N+1.
\end{equation}
\end{lemma}
\begin{proof}
For every triple of distinct points, $c(x,y,z)\le2/\diam\{x,y,z\}$. Fix $x\in E_N$ and decompose pairs $(y,z)$ according to the dyadic size of $\max\{|x-y|,|x-z|\}$. The contribution of the pairs with this maximum in $(t/2,t]$ is at most
\[
 Ct^{-2}\mu_N(B(x,t))^2.
\]
For $t\in[\varepsilon,4]$, Lemma~\ref{lem:area-growth} bounds this by a constant. There are $O(1+\log\varepsilon^{-1})=O(N+1)$ such dyadic scales. For $t<\varepsilon$, the quadratic bound gives $Ct^2/\varepsilon^2$, whose dyadic sum is bounded. Thus the curvature potential is at most $C(N+1)$ at every $x\in E_N$. Integrating against the probability measure $\mu_N$ proves \eqref{eq:natural-curvature-upper}.
\end{proof}

Choose $c_0>0$, depending only on the constants in Lemmas~\ref{lem:area-growth} and~\ref{lem:natural-curvature-upper}, sufficiently small that
\[
 \nu_N:=c_0(N+1)^{-1/2}\mu_N
\]
satisfies both conditions in \eqref{eq:admissible-capacity-measures}. Explicitly, its linear growth constant is at most $Cc_0$, and
\[
 c^2(\nu_N)=c_0^3(N+1)^{-3/2}c^2(\mu_N)
 \le Cc_0^2\nu_N(E_N).
\]
Theorem~\ref{thm:Tolsa} therefore gives
\begin{equation}\label{eq:capacity-lower}
 \gamma(E_N)\gtrsim(N+1)^{-1/2}.
\end{equation}

\subsection{The upper capacity bound}
The upper bound follows from a curvature estimate for arbitrary admissible measures. The geometric input is the following fixed-depth configuration.

\begin{lemma}\label{lem:robust-triple}
If $E$ is not contained in a line, then there is a fixed depth $b\ge1$ and three distinct words $v_1,v_2,v_3\in\cA^b$ such that, for every choice of points $x_j\in D_{v_j}$ with $j\in\{1,2,3\}$,
\begin{equation}\label{eq:robust-triple}
 c(x_1,x_2,x_3)\ge\kappa>0.
\end{equation}
The constants $b$ and $\kappa$ depend only on the system.
\end{lemma}
\begin{proof}
Choose noncollinear points $z_1,z_2,z_3\in E$. By continuity, there are pairwise disjoint neighbourhoods $U_j$ of $z_j$ and $\kappa>0$ such that $c(x_1,x_2,x_3)\ge\kappa$ on $U_1\times U_2\times U_3$. For sufficiently large $b$, the depth-$b$ disk containing each $z_j$ has diameter $2L^{-b}$ and is contained in $U_j$. These three disks have distinct addresses and satisfy \eqref{eq:robust-triple}.
\end{proof}

Consider the grouped system with alphabet $\cA^b$, branching number $J:=L^b$, and contraction factor $J^{-1}$. Its depth-$m$ iterate is $E_{bm}$. For $N\ge b$, monotonicity gives
\[
 \gamma(E_N)\le\gamma(E_{b\lfloor N/b\rfloor}).
\]
It therefore suffices to prove $\gamma(E_{bm})\lesssim m^{-1/2}$ for every integer $m\ge1$. The finitely many depths $N<b$ are bounded by $\gamma(D)$.

Identify the grouped alphabet $\cA^b$ with $\{1,\ldots,J\}$, retaining $v_1,v_2,v_3$ for its three distinguished symbols. For a word $\omega$ in this alphabet, $D_\omega$ denotes the disk obtained by concatenating its original length-$b$ words. Fix an integer $m\ge1$ and a measure $\nu$ supported on $E_{bm}$ satisfying \eqref{eq:admissible-capacity-measures}. Set $M:=\nu(E_{bm})$ and
\[
 m_\omega:=\nu(D_\omega),\qquad \ell_\omega:=J^{-|\omega|}.
\]
The disks of each fixed depth are disjoint, and $\nu$ is supported on the terminal iterate. Thus
\begin{equation}\label{eq:cylinder-mass-additivity}
 m_\omega=\sum_{a=1}^Jm_{\omega a},\qquad
 \sum_{|\omega|=k}m_\omega=M.
\end{equation}
The first identity holds for $|\omega|<m$, and the second for $0\le k\le m$. Applying linear growth to concentric balls of radius decreasing to $\ell_\omega$ gives $m_\omega\le\ell_\omega$.

Define the quadratic mass potentials
\begin{equation}\label{eq:mass-potentials}
 P_k:=\sum_{|\omega|=k}\frac{m_\omega^2}{\ell_\omega},\qquad \forall k\in\{0,\ldots,m\}.
\end{equation}
Then $P_k\le M$. Moreover,
\begin{align}
 P_{k+1}-P_k&=\sum_{|\omega|=k}\Delta_\omega,\label{eq:potential-increment}\\
 \Delta_\omega&:=\frac{J\sum_a m_{\omega a}^2-m_\omega^2}{\ell_\omega}
 =\frac J{\ell_\omega}\sum_{a=1}^J\left(m_{\omega a}-\frac{m_\omega}J\right)^2\ge0.
 \label{eq:variance-identity}
\end{align}
The identity follows from \eqref{eq:cylinder-mass-additivity}.

A node $\omega$ with $|\omega|<m$ is \emph{balanced} if $m_{\omega v_j}\ge m_\omega/(2J)$ for each $j\in\{1,2,3\}$, and \emph{unbalanced} otherwise. At an unbalanced node, at least one child satisfies $|m_{\omega a}-m_\omega/J|\ge m_\omega/(2J)$. Thus \eqref{eq:variance-identity} implies
\[
 \Delta_\omega\ge\frac{m_\omega^2}{4J\ell_\omega}.
\]
Since $m_\omega/\ell_\omega\le1$, this gives
\begin{equation}\label{eq:unbalanced-charge}
 \frac{m_\omega^3}{\ell_\omega^2}\le4J\Delta_\omega
 \qquad\text{for every unbalanced node }\omega.
\end{equation}

Menger curvature obeys the scaling rule
\[
 c(F_\omega(x),F_\omega(y),F_\omega(z))
 =\ell_\omega^{-1}c(x,y,z),
\]
since the similarity $F_\omega$ multiplies circumradii by $\ell_\omega$. For a balanced node $\omega$, this identity and \eqref{eq:robust-triple} give $c(x,y,z)\ge\kappa/\ell_\omega$ on $D_{\omega v_1}\times D_{\omega v_2}\times D_{\omega v_3}$. Consequently,
\begin{equation}\label{eq:balanced-charge}
 \iiint_{D_{\omega v_1}\times D_{\omega v_2}\times D_{\omega v_3}}
 c(x,y,z)^2\,d\nu(x)d\nu(y)d\nu(z)
 \ge\frac{\kappa^2}{8J^3}\frac{m_\omega^3}{\ell_\omega^2}.
\end{equation}
These integration domains are pairwise disjoint as $\omega$ varies over all words of length less than $m$. Incomparable words have disjoint parent disks. If one word is a proper prefix of another, the latter domain places all three points inside a single child of the former, whereas the former domain requires three distinct children. Their integrals therefore sum to at most $c^2(\nu)$.

Summing \eqref{eq:balanced-charge} over balanced nodes and \eqref{eq:unbalanced-charge} over unbalanced nodes yields
\begin{equation}\label{eq:total-cubic-mass}
 \sum_{k=0}^{m-1}\sum_{|\omega|=k}\frac{m_\omega^3}{\ell_\omega^2}
 \le\frac{8J^3}{\kappa^2}c^2(\nu)
      +4J(P_m-P_0)
 \lesssim_{J,\kappa}c^2(\nu)+M.
\end{equation}
At each depth $k$, H\"older's inequality and \eqref{eq:cylinder-mass-additivity} give
\[
 \sum_{|\omega|=k}\frac{m_\omega^3}{\ell_\omega^2}
 =J^{2k}\sum_{|\omega|=k}m_\omega^3\ge M^3.
\]
There are exactly $J^k$ words at that depth, including those of zero mass, and $\sum_{|\omega|=k}\ell_\omega=1$. Since admissibility gives $c^2(\nu)\le M$ and $P_m-P_0\le M$, \eqref{eq:total-cubic-mass} implies
\[
 mM^3\le\left(\frac{8J^3}{\kappa^2}+4J\right)M.
\]
For $M>0$, division by $M$ gives $M\lesssim_{J,\kappa}m^{-1/2}$; the same conclusion holds for $M=0$. Taking the supremum over admissible measures in Theorem~\ref{thm:Tolsa} yields $\gamma(E_{bm})\lesssim m^{-1/2}$. The preceding reduction gives the upper bound in \eqref{eq:general-capacity} for every $N\ge0$.

Together with \eqref{eq:capacity-lower}, this proves \eqref{eq:general-capacity}. For the neighbourhood estimate, Lemma~\ref{lem:neighbourhood-comparison} and monotonicity give
\[
 \gamma(E_{k_+(\delta)})
 \le\gamma(\Nbd{\delta}{E})
 \le\gamma(E_{k_-(\delta)}).
\]
Both indices equal $\log_L(\delta^{-1})+O_{\cF}(1)$, so \eqref{eq:general-capacity} gives \eqref{eq:neighbourhood-capacity}. This completes the proof of Theorem~\ref{thm:general-capacity}.

\begin{remark}
The upper bound requires one fixed noncollinear triple of cylinders. Its constants may depend on the depth and curvature of that triple; no lower bound on the proportion of noncollinear children is required.
\end{remark}

\section{Proofs of the main results and endpoint restrictions}\label{sec:conclusions}
We complete the proofs of the main results and discuss the endpoint restrictions.

\subsection{Noncollinearity and pure unrectifiability of the attractor}\label{subsec:geometric-consequences}
We begin by showing that every system satisfying the hypotheses of Theorem~\ref{thm:structural-projection} has a noncollinear attractor.

\begin{lemma}\label{lem:noncollinearity}
Every system satisfying the hypotheses of Theorem~\ref{thm:structural-projection} has a noncollinear attractor.
\end{lemma}
\begin{proof}
Strong separation implies that $E$ contains at least two distinct points. Suppose that $E$ is contained in an affine line $\ell$. For every symbol $a$, the set $F_a(E)\subset E$ contains two distinct points, so its affine hull is both $F_a(\ell)$ and $\ell$. Thus $F_a(\ell)=\ell$, which implies $\tau_a\in\{0,1/2\}$ modulo one. In particular, $\|2\tau\|_{\T^q,\infty}=0$, contradicting \eqref{eq:BA} with $k=2$.
\end{proof}

Let $E$ be the attractor of such a system. Proposition~\ref{prop:regularity} gives $0<\cH^1(E)<\infty$. Since $E\subset E_N$ for every $N$, Theorem~\ref{thm:structural-projection} implies
\[
 \cH^1(\pi_\theta(E))
 \le\inf_{N\ge0}\cH^1(\pi_\theta(E_N))=0
 \qquad\text{for every }\theta\in[0,\pi).
\]
Hence $\Fav(E)=0$, and the Besicovitch--Federer theorem \cite[Theorem~18.1]{Mattila1995} implies that $E$ is purely $1$-unrectifiable. Lemma~\ref{lem:noncollinearity} permits application of Theorem~\ref{thm:general-capacity}; monotonicity then gives $\gamma(E)\le\inf_N\gamma(E_N)=0$.

\subsection{Proofs of the main theorems}
For either system below, put $t_\delta:=\log\delta^{-1}$ and $N:=\lfloor\log_L\delta^{-1}\rfloor$. Then $\delta\le L^{-N}$ and $N+1\asymp t_\delta$ as $\delta\rightarrow0$. Monotonicity transfers the neighbourhood bounds at scale $L^{-N}$ to scale $\delta$.

\begin{proof}[Proof of Theorem~\ref{thm:explicit-main}]
Section~\ref{sec:explicit} verifies the hypotheses of Corollary~\ref{cor:sixteen-endpoint}. Since $\beta_*>53/100$, that corollary and the preceding scale conversion give \eqref{eq:explicit-main-projection}. Lemma~\ref{lem:noncollinearity} and Theorem~\ref{thm:general-capacity} give \eqref{eq:explicit-main-capacity}. Its lower bound yields
\[
 \Fav(\Nbd{\delta}{E})
 \lesssim t_\delta^{-53/100}
 =\bigl(t_\delta^{-1/2}\bigr)^{53/50}
 \lesssim\gamma(\Nbd{\delta}{E})^{53/50},
\]
which proves \eqref{eq:explicit-main-comparison}. The geometric assertions follow from Subsection~\ref{subsec:geometric-consequences}.
\end{proof}

For this attractor, write $g_\delta:=\gamma(\Nbd{\delta}{E})$. Since $t_\delta\asymp g_\delta^{-2}$, the endpoint estimate also gives
\[
 \Fav(\Nbd{\delta}{E})
 \lesssim g_\delta^{2\beta_*}
       [\log(2+g_\delta^{-1})]^{\beta_*}.
\]
Thus $\Fav(\Nbd{\delta}{E})\lesssim_p g_\delta^p$ for every $p\in(0,2\beta_*)$, where $2\beta_*=1.0623942474\ldots$.

\begin{proof}[Proof of Theorem~\ref{thm:arbitrary-exponent}]
Fix $\alpha\in(0,1)$, choose $q,L$ as in \eqref{eq:q-choice}--\eqref{eq:L-choice}, and take the attractor $E_\alpha$ of \eqref{eq:general-explicit-maps}. Section~\ref{sec:explicit} verifies the hypotheses of Theorem~\ref{thm:structural-projection} with $\beta(q,L)>\alpha$. That theorem and the common scale conversion give \eqref{eq:arbitrary-projection}. The geometric assertions follow from Subsection~\ref{subsec:geometric-consequences}, and Theorem~\ref{thm:general-capacity} gives \eqref{eq:arbitrary-capacity}. The system depends only on the fixed choices $\alpha,q,L$.
\end{proof}

\subsection{Endpoint restrictions}
The Fourier estimate used in the projection argument is \eqref{eq:discrepancy-rate} for the orientation measure \eqref{eq:orientation-measure}. The hypotheses of Theorem~\ref{thm:structural-projection} require rotating maps satisfying \eqref{eq:BA} and therefore exclude the classical four-corner system.

For each fixed $q$, writing $x:=\log_Lr\in(0,1)$ in \eqref{eq:structural-exponent} gives
\[
 \beta=\frac{x}{2+2/q-x}<\frac{q}{q+2}.
\]
Thus exponents approaching one in \eqref{eq:family-exponent} require varying the defining system. This restriction concerns the exponent supplied by the present argument.

For every system in Theorem~\ref{thm:general-capacity}, Proposition~\ref{prop:favard-lower}, Lemma~\ref{lem:neighbourhood-comparison}, and \eqref{eq:neighbourhood-capacity} imply
\[
 \Fav(\Nbd{\delta}{E})\gtrsim(\log\delta^{-1})^{-1}
 \gtrsim\gamma(\Nbd{\delta}{E})^2.
\]
Since $\gamma(\Nbd{\delta}{E})\to0$, an upper bound $\Fav(\Nbd{\delta}{E})\lesssim\gamma(\Nbd{\delta}{E})^p$ with $p>2$ is impossible for these systems. In this class, the endpoint $p=2$ is equivalent to an upper bound $\Fav(\Nbd{\delta}{E})\lesssim(\log\delta^{-1})^{-1}$.

\begin{question}
Does there exist an explicitly specified finite system satisfying Definition~\ref{def:disk-separation}, with a noncollinear attractor $E$, such that
\[
 \Fav(\Nbd{\delta}{E})\lesssim(\log\delta^{-1})^{-1}, \qquad \forall \delta \in (0,1)?
\]
Can $E$ be chosen to satisfy the stronger estimate
\[
 \sup_{0\le\theta<\pi}\cH^1\bigl(\pi_\theta(\Nbd{\delta}{E})\bigr)
 \lesssim(\log\delta^{-1})^{-1}\qquad\forall\delta\in(0,1)?
\]
\end{question}

The first question is an existence version of \cite[Question~7.3]{PS2002}. Such a bound cannot hold for every strongly separated self-similar $1$-set: Bateman--Volberg \cite{BV2010} prove the lower bound $(\log N)/N$ for the classical four-corner iterates. Theorem~\ref{thm:arbitrary-exponent} establishes every exponent strictly below one, with a fixed finite system for each exponent.

\section{The lift to higher dimensions}\label{sec:lift}
Fix $0<\alpha<1$ and integers $1\le k<d$. Let $E_\alpha\subset\R^2$
be the attractor in Theorem~\ref{thm:arbitrary-exponent}, and put
$Q=[0,1]^{k-1}$. In the decomposition
$$
\R^d=\R^2\times\R^{k-1}\times\R^{d-k-1}
$$
define
\begin{equation}\label{eq:liftedset}
 F_{k,\alpha}:=E_\alpha\times Q\times\{0\}.
\end{equation}
We identify a zero-dimensional factor with a singleton. Thus $k=1$
gives an isometric copy of $E_\alpha$, and $d=k+1$ gives the product
$E_\alpha\times Q$ in its natural ambient space. Thus, the set $F_{k,\alpha}$ is obtained from $E_{\alpha}$ (viewed as a subset of the plane spanned by $x_1$ and $x_2$) by attaching a unit $(k-1)$-dimensional box in the $x_3,...,x_{k+1}$ variables to each point of $E_{\alpha}$ .

The following states the basic geometric properties of the set $F_{k,\alpha}$ which are inherited from the set $E_{\alpha}$.

\begin{proposition}\label{prop:lifted-geometry}
The set $F_{k,\alpha}$ is the attractor of a finite homogeneous system
of similarities satisfying the open set condition. It is purely
$k$-unrectifiable and satisfies $0<\cH^k(F_{k,\alpha})<\infty$.
\end{proposition}

\begin{proof}
Write $\cF_\alpha=\{f_1,\ldots,f_L\}$ for the planar system, with common
contraction ratio $L^{-1}$. For $1\le i\le L$ and
$a\in\{0,\ldots,L-1\}^{k-1}$, set
\[
 S_{i,a}(x,y,z):=\bigl(f_i(x),L^{-1}(y+a),L^{-1}z\bigr).
\]
These are $L^k$ similarities of ratio $L^{-1}$. Since the cubes
$L^{-1}(Q+a)$ tile $Q$ and the sets $f_i(E_\alpha)$ cover $E_\alpha$,
the union of the images $S_{i,a}(F_{k,\alpha})$ equals $F_{k,\alpha}$.
The contraction principle gives uniqueness of the nonempty compact
attractor.

The planar disks $f_i(B_2(0,1))$ are disjoint and contained in
$B_2(0,1)$. The open subcubes $L^{-1}((0,1)^{k-1}+a)$ are likewise
disjoint and contained in $(0,1)^{k-1}$. Hence the images of
\[
 O:=B_2(0,1)\times(0,1)^{k-1}\times B_{d-k-1}(0,1)
\]
are pairwise disjoint subsets of $O$ (and we take the convention that $B_0(0,1)$ is the singleton $\{0\}$). The open set condition and the identity $L^k(L^{-1})^k=1$ give $0<\cH^k(F_{k,\alpha})<\infty$. Finally, since $E_\alpha$ is a noncollinear planar subset, the affine hull of $F_{k,\alpha}$ has dimension $2+(k-1)=k+1$. Mattila's rigidity theorem for self-similar sets therefore gives pure $k$-unrectifiability of the attractor $F_{k,\alpha}$ (see \cite[Corollary~4.3 and Remark~4.5]{Mattila1982}).
\end{proof}

The comparison of projections uses only the product structure and the
inclusion of $E_\alpha$ in the closed unit disk. In particular, its
constant is independent of the planar system and the ambient dimension.

\begin{theorem}\label{thm:lifted}
For every $0<\delta<1$,
\begin{equation}\label{eq:liftedthm-1}
 \sup_{V\in\mathcal G(d,k)}
 \cH^k\bigl(P_V(\Nbd{\delta}{F_{k,\alpha}})\bigr)
 \asymp_k
 \sup_{0\le\theta<\pi}
 \cH^1\bigl(\pi_\theta(\Nbd{\delta}{E_\alpha})\bigr).
\end{equation}
In particular,
\begin{equation}\label{eq:lifted-decay}
 \sup_{V\in\mathcal G(d,k)}
 \cH^k\bigl(P_V(\Nbd{\delta}{F_{k,\alpha}})\bigr)
 \lesssim_{k,\alpha}(\log\delta^{-1})^{-\alpha},
 \qquad 0<\delta<1 .
\end{equation}
\end{theorem}

\begin{proof}
Write $E=E_\alpha$, $F=F_{k,\alpha}$ and let $W=\{0\}\times
\R^{k-1}\times\{0\}$ be the subspace parallel to the cube.

\medskip
\noindent
\underline{\textit{The lower bound}}. For a planar unit vector $u \in \mathbb{S}^1$, write $\pi_u(x)=\langle x,u\rangle$. On the product plane $V_u=\R u\oplus W$, the projection of
$\Nbd{\delta}{F}$ contains $\pi_u(\Nbd{\delta}{E})\times Q$.
Indeed, $\Nbd{\delta}{E}\times Q\times\{0\}\subset\Nbd{\delta}{F}$.
Since $Q$ has $(k-1)$-dimensional volume one, taking the supremum
over $u$ gives the lower bound in \eqref{eq:liftedthm-1} with constant one.

\medskip

\noindent
\underline{\textit{The upper bound}}.
For the upper bound, fix $V\in\mathcal G(d,k)$. Since $\dim W=k-1$,
there is a unit vector $e\in V\cap W^\perp$. Its component in the
first two coordinates has the form $tu$ with $0\le t\le1$ and
$u\in\mathbb S^1$ (choose any $u$ if $t=0$). For $(x,y,0)\in F$,
\begin{equation}\label{eq:ecoordinates}
 \langle P_V(x,y,0),e\rangle=t\pi_u(x).
\end{equation}
Thickening $F$ by $\delta$ adds $[-\delta,\delta]$ to this coordinate,
so the set of $\bm{e}$-\textbf{coordinates} (i.e. the set of scalar values in \eqref{eq:ecoordinates} which we just calculated) of $P_V(\Nbd{\delta}{F})$ is exactly $t\pi_u(E)+[-\delta,\delta]$. 

For every nonempty compact
$K\subset \R$, it is clear that
\begin{equation}\label{eq:lift-interval-contraction}
 \cH^1\bigl(tK+[-\delta,\delta]\bigr)
 \le \cH^1\bigl(K+[-\delta,\delta]\bigr),
 \qquad 0\le t\le1.
\end{equation}
Indeed, $K+[-\delta,\delta]$ has finitely many interval components,
each of length at least $2\delta$. If $[a-\delta,b+\delta]$ is one
such component, the points of $K$ in that component lie in $[a,b]$.
Their images under multiplication by $t$, thickened by $\delta$,
lie in $[ta-\delta,tb+\delta]$, of length at most $b-a+2\delta$.
Summing over the components proves \eqref{eq:lift-interval-contraction}.
With $K=\pi_u(E)$, the set on the right is precisely
$\pi_u(\Nbd{\delta}{E})$.

Since $E\subset B_2 (0,1)$ and $\delta<1$, the projected neighbourhood
$B=P_V(\Nbd{\delta}{F})$ lies in the ball of radius $\sqrt{k}+1$
in $V$. In particular, each section 
$$
B\cap(se+H), \quad \textrm{where } H=V\cap e^\perp,
$$
has $(k-1)$-dimensional measure bounded by a
constant depending only on $k$ (and, if $k=1$, each nonempty section is
a single point). Fubini's theorem and \eqref{eq:lift-interval-contraction}
give
\[
 \begin{aligned}
 \cH^k(B)
 &=\int_{t\pi_u(E)+[-\delta,\delta]}
          \cH^{k-1}\bigl(B\cap(se+H)\bigr)\,ds\\[1ex]
 &\lesssim_k\cH^1\bigl(\pi_u(\Nbd{\delta}{E})\bigr).
 \end{aligned}
\]
Taking the supremum over $V$ proves \eqref{eq:liftedthm-1}.
Theorem~\ref{thm:arbitrary-exponent} gives \eqref{eq:lifted-decay}
for all sufficiently small $\delta$. For larger $\delta<1$, the same
radius bound on $B$ gives the estimate after increasing the constant.
\end{proof}

The same slicing argument applies to the Cantor iterates $E_N$ in
\eqref{eq:iterate-definition}. For $F_N=E_N\times Q\times\{0\}$,
the $e$-coordinates of $P_V(F_N)$ form the set $t\pi_u(E_N)$, of length at most
$\cH^1(\pi_u(E_N))$, and the product planes again give the reverse bound.
Applying Theorem~\ref{thm:structural-projection} with the parameters in
\eqref{eq:q-choice}--\eqref{eq:L-choice} gives, for $N\ge1$,
\[
 \sup_{V\in\mathcal G(d,k)}\cH^k(P_V(F_N))
 \asymp_k \sup_{0\le\theta<\pi}\cH^1(\pi_\theta(E_N))
 \lesssim_\alpha N^{-\alpha}.
\]
For a related covering argument, we recommend seeing
\cite[Section~4.2]{MarshallThesis2026}.

\section*{Acknowledgements}
C.M. is supported by an \emph{Arts and Science Postdoctoral Fellowship} and an \emph{NSERC Canadian Postdoctoral Research Award} at the University of Toronto. C.M. thanks Ignacio Uriarte-Tuero for helpful discussions and mentorship.

\end{document}